\documentclass[letterpaper, dvipsnames, 12pt]{article}
\usepackage[margin=1in]{geometry}
\usepackage[utf8]{inputenc}
\usepackage[T1]{fontenc}
\usepackage{amsmath,amsthm,amssymb}

\usepackage{graphicx}
\usepackage{tikz}
\usepackage{hyperref}
\usepackage{color}
\usepackage{enumerate}
\usepackage{multicol}
\usepackage{xcolor}
\usepackage{amsrefs}
\newcommand{\SL}{\sum\limits_}

\newcommand{\R}{\mathbb{R}}
\newcommand{\N}{\mathbb{N}}
\newcommand{\Z}{\mathbb{Z}}
\newcommand{\PR}{\mathbb{P}}
\newcommand{\E}{\mathbb{E}}

\newcommand{\sgn}{\text{sgn}}

\usepackage[autostyle,english=american]{csquotes}
\MakeOuterQuote{"}

\newtheorem{ccounter}{ccounter}[section]
\newtheorem{thm}{Theorem}
\newtheorem{lem}[ccounter]{Lemma}
\newtheorem{cor}[ccounter]{Corollary}
\newtheorem{defn}[ccounter]{Definition}
\newtheorem{prop}[ccounter]{Proposition}

\newtheorem{rem}[ccounter]{Remark}

\title{Limit Profile for the high-temperature Curie-Weiss model}
\author{Lazaros Karageorgiou\thanks{Mandoulides School, Thessaloniki, Greece, Email:karageorgioulazaros@gmail.com} \and Kyprianos-Iason Prodromidis\thanks{Department of Mathematics, Princeton University, Email:kp2702@princeton.edu}}
\date{}
\begin{document}
\maketitle
\begin{abstract}
In this paper, we consider the Ising model on the complete graph, also known as the Curie-Weiss model, and establish the limit profile of the Glauber dynamics in the high-temperature regime. Our strategy is a two-dimensional analog of the method developed by Olesker-Taylor and Schmid for the Bernoulli-Laplace urn: The two-coordinate chain associated to the model evolves near-deterministically until just before the cutoff window, while afterwards it approximates a two-dimensional diffusion.
\end{abstract}
\section{Introduction}

Studying the speed of convergence to equilibrium for Markov chains has become a central theme in probability theory. One of the most interesting phenomena that have emerged in this area is the \textit{cutoff phenomenon}, describing an abrupt transition from far from to near stationarity over a negligible period of time. This behavior has been verified in numerous settings, including card shuffling, random walks on graphs and interacting particle systems. Lately, there have been efforts to understand the nature of this abrupt transition, beyond proving its presence in specific models. For more information on the cutoff phenomenon and latest developments regarding research around it, we refer to \cite{Salez} and references therein.\\\\
In the past five years, there have been efforts to refine cutoff phenomenon results previously known by determining the exact way the chain interpolates between the far-from-equilibrium and the near-equilibrium regimes. This is captured by the \textit{limit profile} of the chain. As notable examples of chains in which the limit profile has been determined, we mention card shuffling (e.g. \cite{Nes}, \cite{Nes-OT1}, \cite{Nes-OT2}, \cite{Teyssier},\cite{Zhang}), random walks on graphs and groups (e.g. \cite{Her-OT}, \cite{Lub-Per}) and exclusion processes (e.g. \cite{Buffetov-Nejjar},\cite{He-Schmid},\cite{Lacoin}).\\\\
In this paper, we study the Glauber dynamics for the high-temperature Ising model on the complete graph, also known as the Curie-Weiss model, and characterize its limit profile. The cutoff phenomenon for this Markov chain was established in \cite{LLP} through analysis of the associated two-coordinate chain. Our approach draws inspiration from the method developed by Olesker-Taylor and Schmid in \cite{OTS} for the Bernoulli-Laplace urn model, which consists of two key steps: First, they show that the position of the birth-and-death chain associated with the model evolves approximately deterministically until it nears the center of mass. Subsequently, they prove that after this point, the trajectory of the chain converges in distribution to an Ornstein-Uhlenbeck process, enabling determination of the limit profile. We show that this technique, with suitable adaptations, applies effectively to the two-coordinate chain in the high-temperature Curie-Weiss model.
\\\\ 
Because of the ability to reduce to the two-coordinate chain, the complete graph provides the simplest version of this problem. In more complicated geometries, even proving cutoff in the high-temperature regime is a significantly harder problem. In the cases of lattices~\cite{L-S-lat} and the sufficiently high-temperature regime on bounded degree graphs~\cite{L-S-gen}, this has been proven by Lubetzky and Sly, using a remarkable method they introduced, called information percolation. However, their method does not imply any limit profile result and thus, this question remains, to our knowledge, open for any other graph geometry, apart from the mean-field case.
\subsection{Definitions and main result}
In this subsection, we introduce the models we will be working with and state the main result. We start with the definitions around mixing times of Markov chains, and what it means for a sequence of Markov chains to exhibit cutoff and to have a specific limit profile.
\begin{defn}[Mixing Time]
\begin{itemize}
\item 
The heat kernel $(H_t)_{t\ge0}$ of a continuous-time Markov chain $(X_t)_{t\ge0}$ on a finite state space $\mathcal{X}$ with transition-rate matrix $Q$ is defined as
$$H_t(x,y):=\PR(X_t=y|X_0=x),\ \text{for every}\ x,y\in\mathcal{X}.$$ 
\item
The $\varepsilon$-mixing time of the chain is defined to be
$$t_\text{mix}(\varepsilon):=\inf\left\{t\ge0:\text{d}(t)\le\varepsilon\right\},$$
where $\pi$ is the stationary distribution of the chain,
\begin{equation}\label{d_TV}
\lVert\mu-\nu\rVert_{\text{TV}}:=\sup\limits_{A\subseteq\mathcal{X}}|\mu(A)-\nu(A)|=\dfrac{1}{2}\SL{x\in\mathcal{X}}|\mu(x)-\nu(x)|
\end{equation}
is the \textbf{total variation distance} between two probability measures $\mu$ and $\nu$ on $\mathcal{X}$ and 
$$\text{d}(t):=\sup\limits_{x\in\mathcal{X}}\lVert H_t(x,\cdot)-\pi\rVert_{\text{TV}}.$$
\end{itemize}
\end{defn}
\begin{defn}[Cutoff and Limit Profile]
\begin{itemize}
\item
We say that a sequence $(X_t^{(n)})_{t\ge0}$ of Markov chains exhibits cutoff at time $t_n$ with window of order $w_n$, if $w_n=o(t_n)$ and
\begin{align*}
\lim\limits_{c\to-\infty}\liminf\limits_{n\to\infty}\text{d}_n(t_n+cw_n)=1,\ \
\lim\limits_{c\to\infty}\limsup\limits_{n\to\infty}\text{d}_n(t_n+cw_n)=0.
\end{align*}
\item 
Let $(X_t^{(n)})_{t\ge0}$ be a sequence of Markov chains that exhibit cutoff at time $t_n$ with window $w_n$. We say that the limit profile of the chain is a function $\Psi:\R\to(0,1)$ if for any $\theta\in\R$,
$$\lim\limits_{n\to\infty}\text{d}_n(t_n+\theta w_n)=\Psi(\theta).$$
\end{itemize}
\end{defn}
We move on to definitions related to the Ising model.
\begin{defn}
\begin{itemize}
\item 
Let $G$ be a graph on a vertex set $V$ and write $u\sim v$ if the vertices $u$ and $v$ are connected by an edge. The Ising model on $G$ with inverse temperature $\beta\ge0$ is the probability measure $\mu_\beta$ on $\{-1,1\}^V$ such that
$$\mu_\beta(\sigma)=\dfrac{1}{Z_{G,\beta}}\cdot\exp\left(\beta\SL{u\sim v}\sigma_u\sigma_v\right).$$
Here, $Z_{G,\beta}$ is a constant which ensures that $\mu_\beta$ is a probability measure.
\item 
The continuous-time Glauber dynamics for the Ising model is a Markov chain on the state space $\{-1,1\}^V$ with transition rates equal to $Q_\beta(\sigma,\tau)=0$ if $\sigma,\tau$ differ in at least two coordinates and
$$Q_\beta(\sigma,\tau)=\dfrac{\mu_\beta(\tau)}{\mu_\beta(\sigma)+\mu_\beta(\tau)},$$ 
otherwise. Alternatively, the chain can be described as follows: Each vertex of the graph gets a rate-1 Poisson clock. When the clock assigned to a vertex $u\in V$ rings, the spin at that vertex is instantly updated according to the Ising model, conditioned on the configuration on rest of the graph.
\item 
It is easy to check that the stationary distribution for the Glauber dynamics is the measure $\mu_\beta$.
\end{itemize}
\end{defn}
In this paper, we work on the case in which $G$ is the complete graph on $n$ vertices and the inverse temperature is equal to $\beta/n$ with for some $\beta\in[0,1)$. We determine the limit profile of the respective Glauber dynamics, by proving the following Theorem:
\begin{thm}\label{main-theorem}
For any $\theta\in\R$, let $t_{n,\theta}=(2(1-\beta))^{-1}\log(n)+\theta$. Then,
$$\text{d}_n(t_{n,\theta})\xrightarrow[n\to\infty]{}\Psi(\theta):=\left\lVert N\left(e^{-(1-\beta)\theta+c(\beta)},\dfrac{1}{1-\beta}\right)-N\left(0,\dfrac{1}{1-\beta}\right)\right\rVert_\text{TV},$$
where 
\begin{equation}\label{c(beta)}
c(\beta):=-\int_0^1\left(\dfrac{\dfrac{\beta t - \tanh(\beta t)}{t^2}}{1 - \dfrac{\tanh(\beta t)}{t}}\right)\ \text{d}t
\end{equation}
is a real constant depending only on $\beta\in[0,1)$.
\end{thm}
In fact, as can be seen from Proposition \ref{first-step}, with our methods, we are able to calculate the limit profile from any possible starting point. Throughout the paper, we will denote by $\pi_n$ the stationary distribution of the model, i.e. the Ising model on the complete graph $K_n$ at inverse temperature $\beta/n$.
\subsection{Reduction to the two-coordinate chain}
Fix a configuration $\sigma_0\in\{-1,1\}^n$ and set $A=\{i\in[n]:\sigma_0(i)=1\}$. For $\sigma\in\{-1,1\}^n$, let 
$$U_{\sigma_0}(\sigma)=\SL{i\in A}\sigma(i)\ \ \text{and}\ \ V_{\sigma_0}(\sigma)=\SL{i\in A^c}\sigma(i).$$
Also, for $t\ge0$, let $U_t=U_{\sigma_0}(X_t)$ and $V_t=V_{\sigma_0}(X_t)$, where $\sigma_0=X_0$ and $(X_t)_{t\ge0}$ is the Glauber dynamics for the Curie-Weiss model. Due to these definitions, $(U_0,V_0)=(|A|,-n+|A|)$. When the choice of $\sigma_0$ is clear or does not matter, we drop the $\sigma_0$ from our notation. It was proven in \cite{LLP} (Lemma 3.4) that in terms of total-variation distance from stationarity, one only has to worry about the two-coordinate chain $(U_t,V_t)$. We repeat that proof, for convenience.
\begin{lem}\label{two-coord-red}
Let $H_t^{(X)},H_t^{(U,V)}$ be the heat kernels for $(X_t)_{t\ge0}$ and $((U_t,V_t))_{t\ge0}$, respectively and $\pi_{U,V}$ be the stationary distribution of the two-coordinate chain. Then
$$\left\lVert H_t^{(X)}(\sigma_0,\cdot)-\pi_n\right\rVert_\text{TV}=\left\lVert H_t^{(U,V)}((|A|,-n+|A|),\cdot)-\pi_{U,V}\right\rVert_\text{TV}$$
\end{lem}
\begin{proof}
For any $u,v$, let $K_{u,v}:=\{\sigma:(U(\sigma),V(\sigma))=(u,v)\}$.\ Due to the symmetry between vertices in $A$ and $A^c$, $H_t^{(X)}(\sigma_0,\sigma)-\pi_X(\sigma)$ is constant over all $\sigma\in K_{u,v}$, for any $u,v$. Hence we get that:
\begin{align*}
\SL{\sigma\in K_{u,v}}\left|H_t^{(X)}(\sigma_0,\sigma)-\pi_X(\sigma)\right|&=\left|\SL{\sigma\in K_{u,v}}H_t^{(X)}(\sigma_0,\sigma)-\pi_X(\sigma)\right|\\&=\left|H_t^{(U,V)}\left((|A|,-n+|A|),(u,v)\right)-\pi_{U,V}(u,v)\right|.
\end{align*}
Summing over all $u,v$ and using (\ref{d_TV}) proves our claim.
\end{proof}
\subsection*{Acknowledgments}
The second author would like to thank Allan Sly for useful discussions on this problem, as well as other related topics.
\section{Analysis of the two-coordinate chain}
The goal of this section is to establish basic properties for the two-coordinate chain. At first, it is important to note that since we are working on the complete graph, for any configuration $\sigma\in\{-1,1\}^n$,
$$\SL{x\sim y}\sigma_x\sigma_y=\dfrac{m(\sigma)^2-n}{2}\ \Rightarrow\mu_{\beta/n}(\sigma)\varpropto\exp\left(\dfrac{\beta}{2n}\cdot m(\sigma)^2\right),$$
where $m(\sigma):=\SL{i\in[n]}\sigma(i)$. Therefore, one can check that the chain $(U_t,V_t)$ has state space 
$$\mathcal{T}_{n,A}=\left\{-|A|,-|A|+2,\dots,|A|\right\}\times\left\{-n+|A|,-n+|A|+2,\dots,n-|A|\right\}$$ 
and transition rates equal to 
\begin{align*} 
& q((u,v), (u+2,v)) 
= \dfrac{|A| - u}{4} \cdot \left(1 + \tanh\left(\frac{\beta}{n} \cdot (u+v+1)\right)\right)  \\
& q((u,v), (u-2,v)) 
= \dfrac{|A| + u}{4} \cdot \left(1 - \tanh\left(\frac{\beta}{n} \cdot (u+v-1)\right)\right)
\\
& q((u,v), (u,v+2)) 
= \dfrac{n-|A| - v}{4} \cdot \left(1 + \tanh\left(\frac{\beta}{n} \cdot (u+v+1)\right)\right)  \\
& q((u,v), (u,v-2)) 
= \dfrac{n-|A| + v}{4} \cdot \left(1 - \tanh\left(\frac{\beta}{n} \cdot (u+v-1)\right)\right).
\end{align*}
It will be important that
\begin{equation}\label{diff-of-q-1}
q((u,v),(u+2,v))-q((u,v),(u-2,v))=-\dfrac{u}{2}+\dfrac{|A|}{2}\tanh\left(\dfrac{\beta}{n}(u+v)\right)+O(1)
\end{equation}
\begin{equation}\label{diff-of-q-2}
q((u,v),(u,v+2))-q((u,v),(u,v-2))=-\dfrac{v}{2}+\dfrac{n-|A|}{2}\tanh\left(\dfrac{\beta}{n}(u+v)\right)+O(1).
\end{equation}
Throughout this section, in order to prove our estimates, we will use the following well-known proposition for continuous-time Markov chains.
\begin{prop} \label{make-it-diff}
Let $(W_t)_{t\ge0}$ be a continuous-time Markov chain with state space $\mathcal{W}$ and transition rate matrix $Q$. Then, for any function $f:\mathcal{W}\to\R$
$$\dfrac{\text{d}}{\text{d}t}\E(f(W_t))=\E((Qf)(W_t)),$$
where of course $(Qf)(x)=\SL{y}q(x,y)f(y)$.
\end{prop}
\subsection{Magnetization chain}
We call the Markov chain $(Y_t)_{t\ge0}$ defined as $Y_t=m(X_t)$ the \textit{magnetization chain}. Working similarly as before, we find that $(Y_t)_{t\ge0}$ has state space $\mathcal{Y}_n=\{-n,-n+2,\dots,n-2,n\}$ and transition rates equal to
$$q_n(k,k+2)=\dfrac{n-k}{4}\cdot\left(1+\tanh\left(\dfrac{\beta}{n}\cdot(k+1)\right)\right)=\dfrac{n-k}{4}\cdot\left(1+\tanh\left(\dfrac{\beta}{n}\cdot k\right)\right)+O(1)$$
$$q_n(k,k-2)=\dfrac{n+k}{4}\cdot\left(1-\tanh\left(\dfrac{\beta}{n}\cdot(k-1)\right)\right)=\dfrac{n+k}{4}\cdot\left(1-\tanh\left(\dfrac{\beta}{n}\cdot k\right)\right)+O(1).$$
In this subsection, our goal is to analyze the chain $(Y_t)_{t\ge0}$.\\\\
Let $\mu=1-\beta>0$ and $f_n:\mathcal{Y}_n\to[0,1]$ defined to be
$$f_n(k)=\sgn(k)\cdot\exp\left[-\int_{\frac{|k|}{n}}^1\dfrac{\mu}{x-\tanh(\beta x)}\ \text{d}x\right].$$
\begin{lem} There exists some functions $R_n^{(1)},\ R_n^{(2)}:\mathcal{Y}_n\to\R$ such that for every $k\in\mathcal{Y}_n$,
$$Q_nf_n(k)=-\mu f_n(k)+R_n^{(1)}(k)$$
and 
$$Q_nf_n^2(k)=-2\mu f_n^2(k)+R_n^{(2)}(k)$$
with $\left|R_n^{(1)}(k)\right|\le c/n$ and $\left|R_n^{(2)}(k)\right|\le c/n$, where $c>0$ is an absolute constant.
\end{lem}
\begin{proof}
Let $g:[-1,1]\to\R$ be the odd function such that
$$g(x)=\exp\left(-\int_x^1\dfrac{\mu}{t-\tanh(\beta t)}\ \text{d}t\right),\ \text{for} \ x>0.$$
We claim that $g$ is differentiable in $[-1,1]$, twice differentiable in $[-1,1]\setminus\{0\}$ and that $g''$ is bounded. Differentiability is obvious at every point except 0, so we focus on that. First, for $x>0$,
\begin{equation}\label{g-int}
\int_x^1\dfrac{\mu}{t-\tanh(\beta t)}\ \text{d}t\ -\int_x^1\dfrac{1}{t}\ \text{d}t\ =\ -\int_x^1\left(\dfrac{\dfrac{\beta t - \tanh(\beta t)}{t^2}}{1 - \dfrac{\tanh(\beta t)}{t}}\right)\ \text{d}t
\end{equation}
and observe that the integrand is $> 0$ and bounded. So, due to (\ref{c(beta)}),
\begin{align*}
&\lim\limits_{x\to0^+}\left( \int_x^1\dfrac{\mu}{t-\tanh(\beta t)}\ \text{d}t\ -\int_x^1\dfrac{1}{t}\ \text{d}t\right) =\ c(\beta)\\ \Rightarrow\ &
\lim\limits_{x\to0^+}\ \dfrac{1}{x}\cdot\exp\left(-\int_x^1\dfrac{\mu}{t - \tanh(\beta t)} \text{d}t\right) = e^{-c(\beta)}\\ \Rightarrow\ &
\lim\limits_{x\to 0^{+}}\ \dfrac{g(x)}{x}\ =\ e^{-c(\beta)}.
\end{align*}
Because $g$ is odd we conclude that $$g'(0)\ =\ e^{-c(\beta)}\in\R.$$
Moreover, for $x>0$,
$$(x - \tanh(\beta x))\cdot g'(x)\ = \ \mu\cdot g(x),$$
and by differentiating,
$$g''(x) = \dfrac{\mu\cdot g(x)}{(x - \tanh(\beta x))^2}\cdot\left(\mu - \dfrac{\cosh(\beta x)^2 - \beta}{\cosh(\beta x)^2}\right)=\dfrac{-\beta(1-\beta)\cdot g(x)}{(x - \tanh(\beta x))^2}\cdot\left(1 -\dfrac{1}{\cosh(\beta x)^2}\right).$$
Since $\lim\limits_{x\to0}\ \dfrac{1}{x^2}\left(1-\dfrac{1}{\cosh(\beta x)^2}\right)\ \in\ (0,\infty) $ and $g(0) = 0$, we conclude that 
$$\lim\limits_{x\to0^+}g''(x)=0.$$
Again using the fact that $g$ is odd we conclude that $g''$ is bounded.
This concludes the proof of our initial claim.\\\\
We can now write $f_n(k) = g(\frac{k}{n})$. Then,
\begin{align*}
f_n(k+2)=g\left(\dfrac{k+2}{n}\right)&=g\left(\dfrac{k}{n}\right)+\dfrac{2}{n}\cdot g'\left(\dfrac{k}{n}\right)+\int_{\frac{k}{n}}^{\frac{k+2}{n}}\left(\dfrac{k+2}{n}-t\right)\cdot g''(t)\ \text{d}t
\\&=f_n(k)+\dfrac{2}{n}\cdot g'\left(\dfrac{k}{n}\right) + O\left(\dfrac{1}{n^2}\right).
\end{align*}
Similarly we have: 
$$f_n(k-2)-f_n(k)=-\dfrac{2}{n}\cdot g'\left(\dfrac{k}{n}\right)+O\left(\dfrac{1}{n^2}\right).$$
Combining these two equalities implies that
\begin{align*}
Q_nf_n(k)&=q_n(k,k+2)\cdot (f_n(k+2)-f_n(k))+q_n(k,k-2)\cdot(f_n(k-2)-f_n(k))\\ &=
\dfrac{2}{n}\cdot g'\left(\dfrac{k}{n}\right)\cdot(q_n(k,k+2)-q_n(k,k-2))+O\left(\dfrac{1}{n}\right)\\&=
g'\left(\dfrac{k}{n}\right)\cdot\left[\left(1-\dfrac{k}{n}\right)\cdot\dfrac{1+\tanh\left(\beta\cdot\frac{k}{n}\right)}{2}-\left(1+\dfrac{k}{n}\right)\cdot\dfrac{1-\tanh(\beta\cdot\frac{k}{n})}{2}\right]+O\left(\dfrac{1}{n}\right)\\&
=-g'\left(\dfrac{k}{n}\right)\cdot\left[\dfrac{k}{n}-\tanh\left(\beta\cdot\dfrac{k}{n}\right)\right]+O\left(\dfrac{1}{n}\right)\\&=-\mu f_n(k)+O\left(\dfrac{1}{n}\right).
\end{align*}
The identity $x^2-y^2=2y(x-y)+(x-y)^2$ and the fact that $f_n(k+2)-f_n(k)=O(n^{-1})$ imply that
\begin{align*}
Q_nf_n^2(k)&\ =q_n(k,k+2)\cdot(f_n(k+2)^2-f_n(k)^2)+q_n(k,k-2)\cdot(f_n(k-2)^2-f_n(k)^2)\\ &\ =2f_n(k)\cdot Q_nf_n(k)+O\left(\dfrac{1}{n}\right)\\ &\ =-2\mu f_n^2(k)+O\left(\dfrac{1}{n}\right),
\end{align*}
which is what we wanted.
\end{proof}
\begin{lem}\label{1st-2nd-mom}
For $t_{n,-C}=(2(1-\beta))^{-1}\log(n)-C$ and $Y_0=2|A|-n$, we have
\begin{align*}
\E(f_n(Y_{t_{n,-C}}))&=\dfrac{e^{\mu C}}{\sqrt{n}}\cdot f_n(2|A|-n)+O\left(\dfrac{1}{n}\right) \ \ \text{and}\\  \E(f_n(Y_{t_{n,-C}})^2)&=\dfrac{e^{2\mu C}}{n}\cdot f_n(2|A|-n)^2 + O\left(\dfrac{1}{n}\right).
\end{align*}
\end{lem}
\begin{proof}
\color{red}
\color{black}
Set $F_n^{(1)}(t)=\E(f_n(Y_t))$ and apply Proposition \ref{make-it-diff} to get
\begin{align*}
\left(F_n^{(1)}(t)\right)'&=\dfrac{\text{d}}{\text{d}t}\E\left(f_n(Y_t)\right)=\E(Q_nf_n(Y_t))=-\mu\E(f_n(Y_t))+\E\left(R_n^{(1)}(Y_t)\right)\\&=-\mu F_n^{(1)}(t)+\E\left(R_n^{(1)}(Y_t)\right).
\end{align*}
Solving the differential equation, we find that for all $t\ge0$,
\begin{align}\label{first-mom}
\E(f_n(Y_t))=F_n^{(1)}(t)&=e^{-\mu t}\cdot\left(f_n(Y_0)+\int_0^te^{\mu u}\cdot\E\left(R_n^{(1)}(Y_u)\right)\ \text{d}u\right)
\\&=e^{-\mu t}\cdot\left(f_n(2|A|-n)+\int_0^t e^{\mu u}\cdot\E\left(R_n^{(1)}(Y_u)\right)\ \text{d}u\right).\notag
\end{align}
Similarly, if we set $F_n^{(2)}(t)=\E\left(f_n(Y_t)^2\right)$, we find that
\begin{align}\label{sec-mom}
\E\left(f_n(Y_t)^2\right)=F_n^{(2)}(t)&=e^{-2\mu t}\cdot\left(f_n(Y_0)^2+\int_0^t e^{2\mu u}\cdot \E\left(R_n^{(2)}(Y_u)\right)\ \text{d}u\right)
\\&=e^{-2\mu t}\cdot\left(f_n(2|A|-n)^2+\int_0^t e^{2\mu u}\cdot \E\left(R_n^{(2)}(Y_u)\right)\ \text{d}u\right).\notag
\end{align}
For $t=t_{n,-C}=(2(1-\beta))^{-1}\log(n)-C$ we observe that:
$$\left|\int_0^{t_{n,-C}} e^{\mu u}\cdot \E\left(R_n^{(1)}(Y_u)\right)\text{d}u\right|\le\left|\dfrac{c}{n}\cdot\int_0^{t_{n,-C}} e^{\mu u}\text{d}u\right|\le\left|\dfrac{c}{n}\cdot\dfrac{1}{\mu}\cdot e^{\mu t_{n,-C}}\right|=O\left(\dfrac{1}{\sqrt{n}\cdot e^{\mu C}}\right)$$\\
and
$$\left|\int_0^{t_{n,-C}} e^{2\mu u}\cdot \E\left(R_n^{(2)}(Y_u)\right)\text{d}u\right|\le\left|\dfrac{c}{n}\cdot\int_0^{t_{n,-C}} e^{2\mu u}\text{d}u\right|\le\left|\dfrac{c}{n}\cdot\dfrac{1}{2\mu}\cdot\dfrac{n}{e^{2\mu C}}\right|=O\left(e^{-2\mu C}\right)$$
Plugging in $t=t_{n,-C}$ we get
\begin{align*}
\E(f_n(Y_{t_{n,-C}}))&=\dfrac{e^{\mu C}}{\sqrt{n}}\cdot f_n(2|A|-n)+O\left(\dfrac{1}{n}\right) \ \ \text{and}\\  \E(f_n(Y_{t_{n,-C}})^2)&=\dfrac{e^{2\mu C}}{n}\cdot f_n(2|A|-n)^2 + O\left(\dfrac{1}{n}\right),
\end{align*}
which is what we wanted.
\end{proof}
\begin{lem}\label{conc-for-mag}
Suppose that the sequence $\sigma_0^{(n)}$ of the initial conditions satisfy the property that $|A_n|/n\to\lambda\in(1/2,1]$. Then, the constant $\alpha(\lambda,C):=e^{\mu C+c(\beta)}\cdot g(2\lambda-1)$ satisfies that for every $\varepsilon,\eta>0$, there exists $C>0$ such that if $n$ is large enough,
$$\PR\left[\left|\dfrac{Y_{t_{n,-C}}}{\sqrt{n}}-a(\lambda,C)\right|\le\varepsilon\cdot a(\lambda,C)\right]\ge1-\eta.$$
\end{lem}
\begin{proof}
Lemma \ref{1st-2nd-mom} implies that if $n$ is large enough,
$$\dfrac{\E\left(f_n(Y_{t_{n,-C}})^2\right)}{\E\left(f_n(Y_{t_{n,-C}})\right)^2}=1+O\left(e^{-2\mu C}\right).$$\\
For $\varepsilon'>0$, consider the event $$A_{\varepsilon'} = \left\{\left|f_n(Y_{t_{n,-C}})-\E(f_n(Y_{t_n,-C}))\right|\le\varepsilon'\cdot\E\left(f_n(Y_{t_{n,-C}})\right)\right\}.$$
By applying Chebyshev inequality we get that for any $\varepsilon'>0$, $\PR[A_{\varepsilon'}]\ge1-O_{\varepsilon'}(e^{-2\mu C})$.\\ 
We now claim that for any $\varepsilon>0$, there exists $\varepsilon'>0$ so that if $n$ is large enough,
$$A_{\varepsilon'}\subseteq\left\{\left|\dfrac{Y_{t_{n,-C}}}{\sqrt{n}}-\alpha(\lambda,C)\right|\le\varepsilon\cdot\alpha(\lambda,C)\right\}.$$
Indeed, suppose $A_{\varepsilon'}$ holds. Then, $\frac{Y_{t_{n,-C}}}{n}=o_{n\to\infty}(1)$ so if $n$ is large enough,
\begin{align*}
&\dfrac{f_n(Y_{t_{n,-C}})}{g'(0)+\varepsilon'}<\dfrac{Y_{t_{n,-C}}}{n}<\dfrac{f_n(Y_{t_{n,-C}})}{g'(0)-\varepsilon'}\\ \Rightarrow\ &\dfrac{(1-\varepsilon')\cdot\E(f_n(Y_{t_{n,-C}}))}{g'(0)+\varepsilon'}<\dfrac{Y_{t_{n,-C}}}{n}<\dfrac{(1+\varepsilon')\cdot\E(f_n(Y_{t_{n,-C}}))}{g'(0)-\varepsilon'}
\\ \Rightarrow\ &\dfrac{1-2\varepsilon'}{e^{-c(\beta)}+\varepsilon'}\cdot e^{\mu C}\cdot g(2\lambda-1) <\dfrac{Y_{t_{n,-C}}}{\sqrt{n}}<\dfrac{1+2\varepsilon'}{e^{-c(\beta)}-\varepsilon'}\cdot e^{\mu C}\cdot g(2\lambda-1),
\end{align*}
since $g$ is continuous and $|A_n|/n\to\lambda$. The Lemma follows.
\end{proof}
\begin{lem}\label{mag-1/2}
Suppose that the sequence $\sigma_0^{(n)}$ of the initial conditions satisfy the property that $|A_n|/n\to1/2$. Then, for every $\eta>0$, there exists (large enough) $M_1>0$ so that for every $C>0$, if $n$ is large enough,
$$\PR\left[\left|\dfrac{Y_{t_{n,-C}}}{\sqrt{n}}\right|\le M_1\right]\ge1-\eta.$$
\end{lem}
\begin{proof}
We use Lemma \ref{1st-2nd-mom}, and more specifically, the second moment estimate. In this case, $f_n(2|A_n|-n)=g\left(2\frac{|A_n|}{n}-1\right)\to0$. Therefore, if $n$ is large enough, due to Markov's inequality,
$$\E\left(f_n(Y_{t_{n,-C}})^2\right)=O\left(\dfrac{1}{n}\right)\ \ \Rightarrow\ \ \PR\left[\left|g\left(\dfrac{Y_{t_{n,-C}}}{n}\right)\right|\cdot\sqrt{n}\ge M_1g'(0)/2\right]\le O\left(\dfrac{1}{M_1^2}\right).$$
Working in a similar way as in Lemma \ref{conc-for-mag}, we find that if $n$ is large enough,
$$\PR\left[\left|\dfrac{Y_{t_{n,-C}}}{\sqrt{n}}\right|\ge M_1\right]=O\left(\dfrac{1}{M_1^2}\right),$$
which is what we intended to prove.
\end{proof}
\subsection{First part of the process}
The goal in this subsection is to derive a concentration bound for the pair $(U_{t_{n,-C}},V_{t_{n,-C}})$ for initial conditions $\sigma_0^{(n)}$ such that $|A_n|/n\to\lambda\in[1/2,1)$. Consider $h_n : \mathcal{T}_{n,A}\to \R$ to be the function 
$$h_n(u,v)=\left(\dfrac{u}{|A|}-\dfrac{v}{n-|A|}\right)^2.$$
\begin{lem}\label{conc-h}
Suppose there exists some $\varepsilon_0>$ for which $|A|\in(\varepsilon_0n,(1-\varepsilon_0)n)$. Then,
$$\E(h_n(U_{t_{n,-C}},V_{t_{n,-C}}))=O\left(\dfrac{1}{n}\right).$$
\end{lem}
\begin{proof}
Observe that $h_n(u\pm2,v)-h_n(u,v)=\pm\dfrac{4}{|A|}\left(\dfrac{u}{|A|}-\dfrac{v}{n-|A|}\right)
+O\left(\dfrac{1}{n^2}\right)$ and a similar relation holds for $h_n(u,v\pm2)-h_n(u,v).$
Calculating $Qh_n$ we get
\begin{align*}
Qh_n(u,v)&=\ q((u,v),(u+2,v))\cdot(h_n(u+2,v)-h_n(u,v))\\&\ +q((u,v),(u-2,v))\cdot(h_n(u-2,v)-h_n(u,v))\\&\ +q((u,v),(u,v+2))\cdot(h_n(u,v+2)-h_n(u,v))\\&\ +q((u,v),(u,v-2))\cdot(h_n(u,v-2)-h_n(u,v))\\&=4\left(\dfrac{u}{|A|^2}-\dfrac{v}{|A|\cdot(n-|A|)}\right)\cdot(q((u,v),(u+2,v))-q((u,v),(u-2,v)))\\&\ +4\left(\dfrac{v}{(n-|A|)^2}-\dfrac{u}{|A|\cdot(n-|A|)}\right)\cdot(q((u,v),(u,v+2))-q((u,v),(u,v-2)))\\&\ +O\left(\dfrac{1}{n}\right)\\&\stackrel{(\ref{diff-of-q-1}), (\ref{diff-of-q-2})}{=}-2h_n(u,v)+O\left(\dfrac{1}{n}\right).
\end{align*}
So, there exists a function $R_n^{(3)}(u,v)$ such that $Qh_n(u,v) = -2h_n(u,v) + R_n^{(3)}(u,v)$ with $\left|R_n^{(3)}(u,v)\right|\le c/n$, for all $(u,v)\in\mathcal{T}_{n,A}$.\\ 
Applying Proposition \ref{make-it-diff} and setting $F_n^{(3)}(t)=\E(h_n(U_t,V_t))$ we find that
\begin{align*}
(F_n^{(3)}(t))'=\dfrac{\text{d}}{\text{d}t}\E(h_n(U_t,V_t))&=\E(Qh_n(U_t,V_t))=-2\E(h_n(U_t,V_t)+R_n^{(3)}(U_t,V_t))\\&=-2F_n^{(3)}(t)+\E(R_n^{(3)}(U_t,V_t)).
\end{align*}
Solving the differential equation we get the for all $t\ge0$,
$$\E(h_n(U_t,V_t))=F_n^{(3)}(t)=e^{-2t}\cdot\left(4+\int_0^t e^{2s}\cdot\E\left(R_n^{(3)}(U_s,V_s)\right))\ \text{d}s\right).$$
Working as in the proof of Lemma \ref{1st-2nd-mom} we get that
$$\E(h_n(U_{t_{n,-C}},V_{t_{n,-C}}))=O\left(\dfrac{1}{n}\right).$$
\end{proof}
\begin{prop}\label{conc-for-pair}
Suppose that the sequence $\sigma_0^{(n)}$ of the initial conditions satisfy the property that $|A_n|/n\to\lambda\in(1/2,1)$. Let $\alpha_1(\lambda,C):=\lambda\cdot\alpha(\lambda,C)$ and $\alpha_2(\lambda,C):=(1-\lambda)\cdot\alpha(\lambda,C)$. Then, for every $\varepsilon,\eta>0$, there exists $C>0$ (large enough) such that if $n$ is large enough,
$$\PR\left[\left\{\left|\dfrac{U_{t_{n,-C}}}{\sqrt{n}}-\alpha_1(\lambda,C)\right|\le\varepsilon\cdot\alpha_1(\lambda,C)\right\}\cap\left\{\left|\dfrac{V_{t_{n,-C}}}{\sqrt{n}}-\alpha_2(\lambda,C)\right|\le\varepsilon\cdot\alpha_2(\lambda,C)\right\}\right]\ge1-\eta.$$
\end{prop}
\begin{proof}
In this proof, we always assume that $n$ is large enough. For $\varepsilon>0$ consider the event
$$B_{\varepsilon}=\left\{\left\{\left|\dfrac{Y_{t_{n,-C}}}{\sqrt{n}}-\alpha(\lambda,C)\right|\le\dfrac{\varepsilon}{4}\cdot\alpha(\lambda,C)\right\}\cap\left\{\left|\dfrac{U_{t_{n,-C}}}{|A|}-\dfrac{V_{t_{n,-C}}}{n-|A|}\right|\le\dfrac{\varepsilon}{4\sqrt{n}}\cdot\alpha(\lambda,C)\right\}\right\}.$$
Lemma \ref{conc-h}, together with Markov's inequality for $h_n(U_{t_{n,-C}},V_{t_{n,-C}})$ implies
$$\PR\left[\left|\dfrac{U_{t_{n,-C}}}{|A|}-\dfrac{V_{t_{n,-C}}}{n-|A|}\right|\le\dfrac{\varepsilon}{4\sqrt{n}}\cdot\alpha(\lambda,C)\right]\ge1-O_{\varepsilon}\left(e^{-2\mu C}\right).$$
Hence due to Lemma \ref{conc-for-mag}, for any $\varepsilon,\eta>0$ there exists large enough $C>0$ such that $\PR[B_{\varepsilon}]\ge1-\eta.$
Our main claim is that
$$B_{\varepsilon}\subseteq\left\{\left|\dfrac{U_{t_{n,-C}}}{\sqrt{n}}-\alpha_1(\lambda,C)\right|\le\varepsilon\cdot\alpha_1(\lambda,C)\right\}\cap\left\{\left|\dfrac{V_{t_{n,-C}}}{\sqrt{n}}-\alpha_2(\lambda,C)\right|\le\varepsilon\cdot\alpha_2(\lambda,C)\right\}.$$
We prove only the first inclusion as the other one will follow similarly. Suppose that $B_{\varepsilon}$ holds and $n$ is such that
$$\left(1-\dfrac{\varepsilon}{4}\right)\lambda\le\dfrac{|A_n|}{n}\le\left(1+\dfrac{\varepsilon}{4}\right)\lambda.$$
Observe that
$$\dfrac{U_{t_{n,-C}}}{\sqrt{n}}=\dfrac{|A_n|}{n}\cdot\left(\dfrac{Y_{t_{n,-C}}}{\sqrt{n}}+\dfrac{n-|A_n|}{\sqrt{n}}\cdot\left(\dfrac{U_{t_{n,-C}}}{|A_n|}-\dfrac{V_{t_{n,-C}}}{n-|A_n|}\right)\right).$$
Then
\begin{align*}
\left|\dfrac{U_{t_{n,-C}}}{\sqrt{n}}-\alpha_1(\lambda,C)\right|\le\ &\dfrac{|A_n|}{n}\cdot\left(\left|\dfrac{Y_{t_{n,-C}}}{\sqrt{n}}-\alpha(\lambda,C)\right|+\left|\dfrac{n-|A_n|}{\sqrt{n}}\cdot\left(\dfrac{U_{t_{n,-C}}}{|A_n|}-\dfrac{V_{t_{n,-C}}}{n-|A_n|}\right)\right|\right)\\&+\left|\dfrac{|A_n|}{n}-\lambda\right|\cdot\alpha(\lambda,C)\\\le\ & \left(1+\dfrac{\varepsilon}{4}\right)\lambda\cdot\left(\dfrac{\varepsilon}{4}\cdot\alpha(\lambda,C)+\dfrac{\varepsilon}{4}\cdot\alpha(\lambda,C)\right)+\dfrac{\varepsilon}{4}\lambda\cdot\alpha(\lambda,C)\\\le\ & \varepsilon\cdot\alpha_1(\lambda,C)
\end{align*}
which is what we wanted.
\end{proof}
\begin{prop}\label{conc-pair-1/2}
Let $\lambda=1/2$. For any $\eta>0$ there exists a constant $M_2>0$ such that for any $C>0$, if $n$ is large enough,
$$\PR\left[\left\lVert\left(\dfrac{U_{t_{n,-C}}}{\sqrt{n}},\dfrac{V_{t_{n,-C}}}{\sqrt{n}}\right)\right\rVert_\infty\le M_2\right]\ge1-\eta.$$
\end{prop}
\begin{proof}
Denote by $B_{M_2}$ the event
$$B_{M_2}=\left\{\left|\dfrac{Y_{t_{n,-C}}}{\sqrt{n}}\right|\le \dfrac{M_2}{2}\right\}\cap\left\{\left|\dfrac{U_{t_{n,-C}}}{|A|}-\dfrac{V_{t_{n,-C}}}{n-|A|}\right|\le\dfrac{M_2}{2\cdot\sqrt{n}}\right\}.$$
Due to Lemma \ref{mag-1/2} and similarly as we did before by applying Markov we get that for any $\eta>0$ there exist a large enough constant $M_2$ such that for any $C>0$, if $n$ is large enough,
$$\PR\left[B_{M_2}\right]\ge1-\eta.$$
Now we claim that $$B_{M_2}\subseteq\left\{\left\lVert\left(\dfrac{U_{t_{n,-C}}}{\sqrt{n}},\dfrac{V_{t_{n,-C}}}{\sqrt{n}}\right)\right\rVert_\infty\le M_2\right\}.$$
Applying the triangle's inequality in a similar way as before, we get that
$$\left|\dfrac{U_{t_{n,-C}}}{\sqrt{n}}\right|\le\dfrac{|A_n|}{n}\cdot\left(\left|\dfrac{Y_{t_{n,-C}}}{\sqrt{n}}\right|+\left|\dfrac{n-|A_n|}{\sqrt{n}}\cdot\left(\dfrac{U_{t_{n,-C}}}{|A_n|}-\dfrac{V_{t_{n,-C}}}{n-|A_n|}\right)\right|\right)\le M_2.$$
In a similar way, we can obtain that on $B_{M_2}$,
$\left|\dfrac{V_{t_{n,-C}}}{\sqrt{n}}\right|\le M_2.$
\end{proof}
\subsection{Two useful lemmas}
Next, we state and prove two lemmas that will help us finish the proof of Theorem \ref{main-theorem}. In both cases, we assume that $|A|\in(\varepsilon_0n,(1-\varepsilon_0)n)$ and that the chain $(U_t,V_t)$ does not start at $(|A|,-n+|A|)$, but rather at $(u_0,v_0)\in(-c\sqrt{n},c\sqrt{n})\times(-c\sqrt{n},c\sqrt{n})$.
\begin{lem}\label{tech-1}
For every $a,t_0>0$,
\begin{align*}
\PR\left(|U_{t_0}-u_0|>a\sqrt{n}\right)&=O\left(\dfrac{t_0}{a^2}\right)\ \ \text{and}\\
\PR\left(|V_{t_0}-v_0|>a\sqrt{n}\right)&=O\left(\dfrac{t_0}{a^2}\right),
\end{align*}
where the implied constant depends on $c>0$.
\end{lem}
\begin{proof}
First, note that due to (\ref{g-int}), for any $x\in[-1,1]$,
$$\tanh(\beta x)^2\le x^2\le g(x)^2.$$
We apply Proposition \ref{make-it-diff} to the function $f(u,v)=(u-u_0)^2$:
\begin{align*}
Qf(u,v)&=q((u,v),(u+2,v))(4(u-u_0)+O(1))+q((u,v),(u-2,v))(-4(u-u_0)+O(1))
\\&=2(u-u_0)\left(-u+|A|\tanh\left(\dfrac{\beta}{n}(u+v)\right)\right)+O(n)
\\&=-2(u-u_0)^2+2(u-u_0)\left(|A|\tanh\left(\dfrac{\beta}{n}(u+v)\right)-u_0\right)+O(n)
\\&\le\dfrac{1}{2}\left(|A|\tanh\left(\dfrac{\beta}{n}(u+v)\right)-u_0\right)^2+O(n)
\\&\le|A|^2\tanh\left(\dfrac{\beta}{n}(u+v)\right)^2+u_0^2+O(n)
\\&\le|A|^2f_n(u+v)^2+O(n)\le n^2f_n(u+v)^2+O(n).
\end{align*}
Therefore, due to (\ref{sec-mom}),
\begin{align*}
\dfrac{\text{d}}{\text{d}t}\E\left((U_t-u_0)^2\right)&\le n^2\cdot\E\left(f_n(U_t+V_t)^2\right)+O(n)\\&\le n^2\cdot\left(f_n(u_0+v_0)^2+O\left(\dfrac{1}{n}\right)\right)+O(n)
\\&=O(n).
\end{align*}
So, for any $t>0$,
$$\E\left((U_t-u_0)^2\right)=O(tn)\ \Rightarrow\ \PR(|U_{t_0}-u_0|>a\sqrt{n})\le\dfrac{\E\left((U_{t_0}-u_0)^2\right)}{a^2n}=O\left(\dfrac{t_0}{a^2}\right).$$
The proof for the second relation is identical.
\end{proof}
\begin{lem}\label{tech-2}
For every $t_0>0$,
$$\PR\left(\exists t\le t_0:|U_t|\ \text{or}\ |V_t|\ > \varepsilon_0n/3\right)=o_{n\to\infty}(1).$$
\end{lem}
\begin{proof}
Divide the interval $[0,t_0]$ into intervals $I_1,I_2,\dots,I_k$ of lengths $\ell_i\in[\varepsilon_0/100,\varepsilon_0/50]$. Let $E$ be the event that in each interval $I_i$, the number of updates that occur in $I_i$ is $\le\varepsilon_0n/25$. Due to standard large deviation estimates, $\PR(E^c)=o_{n\to\infty}(1)$. Moreover, observe that on $E$, the number of updates on any time interval of length $\le\varepsilon_0/50$ is at most $\varepsilon_0n/6$, as the interval is covered by the union of at most 3 of the intervals $I_i$. Let
$$Z=\int_0^{t_0}\mathbf{1}_{|U_t|\ge\varepsilon_0n/6}\ \text{d}t.$$
It is easy to observe that on the event $\{\exists t\le t_0:|U_t|>\varepsilon_0n/3\}\cap E$, $Z\ge\varepsilon_0/50$. Also,
$$\E(Z)=\int_0^{t_0}\PR(|U_t|\ge\varepsilon_0n/6)\ \text{d}t\le\int_0^{t_0}\dfrac{ct}{\varepsilon_0^2n}\ \text{d}t=O\left(n^{-1}\right).$$
Therefore, due to Markov's inequality,
\begin{align*}
\PR(\exists\ t\le t_0:|U_t|>\varepsilon_0n/3)&\le\PR(E^c)+\PR(\{\exists t\le t_0:|U_t|>\varepsilon_0n/3\}\cap E)\\&\le\PR(E^c)+\PR(Z\ge\varepsilon_0/50)\le\PR(E^c)+O(n^{-1})\\&=o_{n\to\infty}(1).
\end{align*}
Proving the statement in question is identical in the case of $(V_t)_{t\ge0}$, which implies the desired result.
\end{proof}
\section{Diffusion Approximation and TV distances}
In this section, we work with the rescaled version $(\tilde{U}_t,\tilde{V}_t)$ of $(U_t,V_t)$, defined as 
$$\tilde{U}_t^{(n)}=\dfrac{U_{t_{n,-C}+t}}{\sqrt{n}}\ \ \text{and}\ \ \tilde{V}_t^{(n)}=\dfrac{V_{t_{n,-C}+t}}{\sqrt{n}}.$$
The process $((\tilde{U}_t^{(n)},\tilde{V}_t^{(n)}))_{t\ge0}$ is a Markov chain with state space 
$$\mathcal{S}_n=\left\{-\dfrac{|A|}{\sqrt{n}},-\dfrac{|A|-2}{\sqrt{n}},\dots,\dfrac{|A|}{\sqrt{n}}\right\}\times\left\{-\dfrac{n-|A|}{\sqrt{n}},-\dfrac{n-|A|-2}{\sqrt{n}},\dots,\dfrac{n-|A|}{\sqrt{n}}\right\}$$ 
and transition rates equal to
\begin{align*} 
& q\left((\tilde{u},\tilde{v}), \left(\tilde{u}+\dfrac{2}{\sqrt{n}},\tilde{v}\right)\right) 
= \dfrac{|A| - \tilde{u}\sqrt{n}}{4} \cdot \left(1 + \tanh\left(\frac{\beta}{\sqrt{n}} \cdot (\tilde{u}+\tilde{v})\right)\right) + O(1)  \\
& q\left((\tilde{u},\tilde{v}), \left(\tilde{u}-\dfrac{2}{\sqrt{n}},\tilde{v}\right)\right) 
= \dfrac{|A| + \tilde{u}\sqrt{n}}{4} \cdot \left(1 - \tanh\left(\frac{\beta}{\sqrt{n}} \cdot (\tilde{u}+\tilde{v})\right)\right) + O(1)  \\
& q\left((\tilde{u},\tilde{v}), \left(\tilde{u},\tilde{v}+\dfrac{2}{\sqrt{n}}\right)\right) 
= \dfrac{n-|A| - \tilde{v}\sqrt{n}}{4} \cdot \left(1 + \tanh\left(\frac{\beta}{\sqrt{n}} \cdot (\tilde{u}+\tilde{v})\right)\right) + O(1) \\
& q\left((\tilde{u},\tilde{v}), \left(\tilde{u},\tilde{v}-\dfrac{2}{\sqrt{n}}\right)\right) 
= \dfrac{n-|A| + \tilde{v}\sqrt{n}}{4} \cdot \left(1 - \tanh\left(\frac{\beta}{\sqrt{n}} \cdot (\tilde{u}+\tilde{v})\right)\right) + O(1)
\end{align*}
\subsection{Convergence to a Diffusion}
In this subsection, we prove that the rescaled version $(\tilde{U}_t^{(n)},\tilde{V}_t^{(n)})$ converges in distribution to an Ornstein-Uhlenbeck process.
\begin{prop}\label{converg-diff}
Assume that $|A_n|/n\to\lambda\in[1/2,1)$ and $(\tilde{U}_0^{(n)},\tilde{V}_0^{(n)})=(\tilde{u}_n,\tilde{v}_n)\to (u,v)$. Then, for all $t\ge0$,
$$\begin{pmatrix}
\tilde{U}_t^{(n)}\\
\tilde{V}_t^{(n)}
\end{pmatrix}\xrightarrow[n\to\infty]{(d)}D_t\sim N(\mu_t,\Sigma_t),$$
where
$$\mu_t=e^{tA}\begin{pmatrix}
u\\v
\end{pmatrix}\ \ \text{and}\ \ \Sigma_t=\int_0^te^{sA}BB^{\top}e^{sA^\top}\ \text{d}s,$$
and the matrices $A,B$, respectively are
\begin{equation}\label{def-A-B}
A=\begin{pmatrix}
-1+\lambda\beta&\lambda\beta\\(1-\lambda)\beta&-1+(1-\lambda)\beta
\end{pmatrix}\ \ \text{and}\ \ 
B=\begin{pmatrix}
\sqrt{2\lambda}&0\\0&\sqrt{2(1-\lambda)}
\end{pmatrix}.
\end{equation}
\end{prop}
In the proof of Proposition \ref{converg-diff}, we will use the following Theorem from \cite{Dur}:
\begin{thm}[Theorem 8.7.1 from \cite{Dur}]\label{conv-to-diff}
For each $n\in\N$, let $(Z_t^{(n)})_{t\ge0}$ be a continuous-time Markov chain on a finite state space $\mathcal{S}_n\subseteq\R^d$ with transition-rate matrices $(q_n(x,y))_{x,y\in\mathcal{S}_n}$. Suppose there are Lipschitz functions $b:\R^d\to\R^d, \sigma:\R^d\to\R^{d\times d}$ such that: For any sequence $x_n\in\mathcal{S}_n$ such that $x_n\to x\in\R^d$,
\begin{enumerate}
\item 
$\lim\limits_{n\to\infty}\SL{y\in\mathcal{S}_n}(y-x_n)\cdot q(x_n,y)\mathbf{1}_{|x_n-y|\le1}=b(x).$
\item 
$\lim\limits_{n\to\infty}\SL{y\in\mathcal{S}_n}(y-x_n)(y-x_n)^\top\cdot q_n(x_n,y)\mathbf{1}_{|x_n-y|\le1}=\sigma(x)\sigma(x)^\top.$
\item 
For all $R>0$ and $\varepsilon>0$,
$$\sup\limits_{|x|\le R}\ \SL{y:|y-x|\ge\varepsilon}q_n(x_n,y)\xrightarrow[n\to\infty]{}0.$$
\end{enumerate}
If $Z_0^{(n)}=z_n\to z$, then, for all $T>0$,
$$(Z_t^{(n)})_{t\in[0,T]}\xrightarrow[n\to\infty]{d}(D_t)_{t\in[0,T]},$$
where $(D_t)_{t\ge0}$ is the solution to the martingale problem
\begin{equation}\label{mtg-pb}
\text{d}D_t=b(D_t)\text{d}t+\sigma(D_t)\text{d}B_t
\end{equation}
with $D_0=z$ and $(B_t)_{t\ge0}$ being a $d$-dimensional typical Brownian motion.
\end{thm}
\begin{proof}[Proof of Proposition \ref{converg-diff}]
We prove that the process $((\tilde{U}_t^{(n)},\tilde{V}_t^{(n)}))_{t\ge0}$ satisfies conditions 1-3 of Theorem \ref{conv-to-diff} for the functions $b(x)=Ax$ and $\sigma(x)=B$, where $A$ and $B$ are the matrices defined in (\ref{def-A-B}).  Suppose $z_n=(\tilde{u}_n,\tilde{v}_n)$ is a sequence of 2-dimensional vectors with $(\tilde{u}_n,\tilde{v}_n)\to (\tilde{u},\tilde{v})=z$.\\\\
For condition 1, we have
\begin{align*}
&\sum_{y \in \mathcal{S}_n} 
(y - z_n) \cdot q(z_n, y) \mathbf{1}_{|z_n - y| \le 1}\\=\ & \left(\dfrac{2}{\sqrt{n}},0\right) \cdot\left(q\left((\tilde{u}_n,\tilde{v}_n),\left(\tilde{u}_n+\dfrac{2}{\sqrt{n}},\tilde{v}_n\right)\right) -q\left((\tilde{u}_n,\tilde{v}_n),\left(\tilde{u}_n-\dfrac{2}{\sqrt{n}},\tilde{v}_n\right)\right)\right)\\ +&\left(0,\dfrac{2}{\sqrt{n}}\right)\cdot\left(q\left((\tilde{u}_n,\tilde{v}_n),\left(\tilde{u}_n,\tilde{v}_n+\dfrac{2}{\sqrt{n}}\right)\right) -q\left((\tilde{u}_n,\tilde{v}_n),\left(\tilde{u}_n,\tilde{v}_n-\dfrac{2}{\sqrt{n}}\right)\right)\right) \\\\=\ &\begin{pmatrix}\dfrac{|A_n|}{\sqrt{n}}\cdot\tanh\left(\dfrac{\beta}{\sqrt{n}}\cdot(\tilde{u}_n+\tilde{v}_n)\right)-\tilde{u}_n+O\left(\dfrac{1}{\sqrt{n}}\right)\\\\\dfrac{n-|A_n|}{\sqrt{n}}\cdot\tanh\left(\dfrac{\beta}{\sqrt{n}}\cdot(\tilde{u}_n+\tilde{v}_n)\right)-\tilde{v}_n+O\left(\dfrac{1}{\sqrt{n}}\right)\end{pmatrix}\\\\\xrightarrow[n\to\infty]{}&\begin{pmatrix}\lambda\beta\cdot(\tilde{u}+\tilde{v})-\tilde{u}\\\\(1-\lambda)\beta\cdot(\tilde{u}+\tilde{v})-\tilde{v}\end{pmatrix}=Az.
\end{align*}
For condition 2, we have
\begin{align*}
&\sum_{y \in \mathcal{S}_n} 
(y - z_n)(y - z_n)^\top \cdot q(z_n, y) \mathbf{1}_{|z_n - y| \le 1}\\=\ & 
\begin{pmatrix}
\frac{4}{n} & 0\\
0 & 0\\
\end{pmatrix}
 \cdot\left(q\left((\tilde{u}_n,\tilde{v}_n),\left(\tilde{u}_n+\dfrac{2}{\sqrt{n}},\tilde{v}_n\right)\right) +q\left((\tilde{u}_n,\tilde{v}_n),\left(\tilde{u}_n-\dfrac{2}{\sqrt{n}},\tilde{v}_n\right)\right)\right)\\ +&\begin{pmatrix}
0 & 0\\
0 & \frac{4}{n}\\
\end{pmatrix}\cdot\left(q\left((\tilde{u}_n,\tilde{v}_n),\left(\tilde{u}_n,\tilde{v}_n+\dfrac{2}{\sqrt{n}}\right)\right)+q\left((\tilde{u}_n,\tilde{v}_n),\left(\tilde{u}_n,\tilde{v}_n-\dfrac{2}{\sqrt{n}}\right)\right)\right) \\\\=\ &
\begin{pmatrix} 
\frac{2}{n}\left(|A_n|-\tilde{u}_n\tanh\left(\frac{\beta (\tilde{u}_n+\tilde{v}_n)}{ \sqrt{n}}\right) + O(1)\right) & 0 \\\\
0 & \frac{2}{n}\left(n-|A_n|-\tilde{v}_n\tanh\left(\frac{\beta (\tilde{u}_n+\tilde{v}_n)}{ \sqrt{n}}\right) + O(1)\right)
\end{pmatrix}\\\xrightarrow[n\to\infty]{}\ &\begin{pmatrix}
2\lambda & 0\\
0 & 2(1-\lambda)\\
\end{pmatrix}=BB^\top.
\end{align*}
Condition 3 is satisfied trivially.\\\\
It now suffices to understand the distribution of the strong solution $(D_t)_{t\ge0}$ at each time $t$, for the specific $b,\sigma$ that we found. Since this diffusion is a two-dimensional Ornstein-Uhlenbeck process, it is well-known (see, for example, \cite{KS}, chapter 5.6) that the distribution at time $t>0$, if the starting point is $z=D_0\in\R^2$, is $N(\mu_t,\Sigma_t)$, with
$$\mu_t=e^{tA}D_0\ \ \text{and}\ \ \Sigma_t=\int_0^te^{sA}BB^{\top}e^{sA^\top}\ \text{d}s.$$
The proof is complete.
\end{proof}
\subsection{TV distance of Gaussians}
The goal of this subsection is to understand the total variation distance between Gaussian random vectors and perform calculations that will be important later. From now on, we will denote by $\Sigma$ the covariance matrix $\Sigma_\infty:=\lim\limits_{t\to\infty}\Sigma_t$.
\begin{lem}\label{small-tv}
\begin{enumerate}
\item 
Let $z,z'\in\R^2$ be two vectors such that $\lVert z-z'\rVert_2\le c_1\cdot\varepsilon\cdot e^{(1-\beta)C}$, for some absolute constant $c_1>0$. Then, for every $\eta,\theta_0>0$, there exist $\varepsilon,C>0$ so that
$$\sup\limits_{|\theta|\le\theta_0}\left\lVert N\left(e^{(C+\theta)A}z',\Sigma_{C+\theta}\right)-N\left(e^{(C+\theta)A}z,\Sigma\right)\right\rVert_{\text{TV}}<\eta.$$
\item 
Let $z,z'\in\R^2$ be two vectors such that $\lVert z-z'\rVert_2\le2M_2$. Then, for every $\eta,\theta_0,M_2>0$, there exists $C>0$ so that
$$\sup\limits_{|\theta|\le\theta_0}\left\lVert N(e^{(C+\theta)A}z',\Sigma_{C+\theta})-N(e^{(C+\theta)A}z,\Sigma)\right\rVert_{\text{TV}}<\eta.$$
\item 
Let $\mu_1,\mu_2\in\R^2$. Then,
$$\left\lVert N(\mu_1,\Sigma)-N(\mu_2,\Sigma)\right\rVert_{\text{TV}}\le c\lVert\mu_2-\mu_1\rVert_{\text{TV}}.$$
\end{enumerate}
\end{lem}
Before we move on to the proof of this Lemma, we mention and prove some linear algebra facts that will be crucial for this proof.\\\\
The eigenvalues of $A$ are $-1$ and $-(1-\beta)$, which are distinct, therefore it is diagonalizable. Set $A=P\cdot\text{diag}(-(1-\beta),-1)\cdot P^{-1}$. Then, $e^{tA}=P\cdot\text{diag}(e^{-(1-\beta)t},e^{-t})\cdot P^{-1}$, which implies that
\begin{equation}\label{norm-of-exp}
\lVert e^{At}\rVert_{\text{OP}}\le Ke^{-(1-\beta)t},
\end{equation}
for some constant $K>0$, for any $t>0$. Using this inequality, we can prove that $\Sigma_{C+\theta}$ and $\Sigma$ are close:
\begin{align*}
\lVert\Sigma_{C+\theta}-\Sigma\rVert_{\text{OP}}&\le\int_{C+\theta}^\infty\left\lVert e^{sA}\right\rVert_{\text{OP}}\cdot\lVert BB^\top\rVert_{\text{OP}}\cdot\left\lVert e^{sA^\top}\right\rVert_{\text{OP}}\ \text{d}s
\\&\le c_2\int_{C+\theta}^\infty K^2e^{-2(1-\beta)s}\ \text{d}s\\&=O(e^{-2(1-\beta)C}).
\end{align*}
Now, let $Y\in\R^{2\times2}$ be a matrix with $\lVert Y\rVert_{\text{OP}}=\rho<1/2$. Then, both eigenvalues $\lambda_1,\lambda_2$ of $Y$ are, in absolute value, at most $\rho$, therefore,
\begin{equation}\label{ineq-tr-det}
|\text{tr}(Y)|=|\lambda_1+\lambda_2|\le2\rho\ \ \text{and}\ \ |\det(I+Y)-1|=|\lambda_1+\lambda_2-\lambda_1\lambda_2|\le3\rho.
\end{equation}
\begin{proof}[Proof of Lemma \ref{small-tv}]
We prove the first statement for a fixed $\theta$, but it is not hard to see that if $\theta$ is restricted to a bounded set in $\R$, everything we do holds uniformly in this bounded set.\\\\
For two distributions $P,Q$ on $\R^2$, let $D_{KL}(P||Q)$ be the Kullback-Leibler divergence from $P$ to $Q$, defined as
$$D_{KL}(P||Q):=\int_{\R^2}p(x)\log\dfrac{p(x)}{q(x)}\ \text{d}x.$$
If the $P,Q$ are Gaussian vectors, $P=N(\mu_1,\Sigma_1), Q=N(\mu_2,\Sigma_2)$, we know that
\begin{equation}\label{KL-Gaus}
D_{KL}(P||Q)=\dfrac{1}{2}\left(\text{tr}(\Sigma_2^{-1}\Sigma_1)-2+(\mu_2-\mu_1)^\top\Sigma_2^{-1}(\mu_2-\mu_1)-\log\dfrac{\det(\Sigma_1)}{\det(\Sigma_2)}\right).
\end{equation}
We apply this equality for $\mu_1\to e^{(C+\theta)A}z', \mu_2\to e^{(C+\theta)A}z, \Sigma_1\to \Sigma_{C+\theta}, \Sigma_2\to\Sigma$, and bound the terms one-by-one.\\\\
Let $\Delta=\Sigma_{C+\theta}-\Sigma$, with $\lVert\Delta\rVert_{\text{OP}}=O(e^{-2(1-\beta)C})$, as we have already explained. Then, keeping (\ref{ineq-tr-det}) in mind,
$$|\text{tr}(\Sigma^{-1}\Sigma_{C+\theta})-2|=|\text{tr}(\Sigma^{-1}(\Sigma+\Delta))-2|=|\text{tr}(\Sigma^{-1}\Delta)|\le2\lVert\Sigma^{-1}\Delta\rVert_{\text{OP}}=O(e^{-2(1-\beta)C}),$$
because the operator norm is sub-multiplicative.\\\\
Also, using (\ref{ineq-tr-det}) and working similarly as above, if $C$ is large enough,
$$\left|\dfrac{\det(\Sigma_{C+\theta})}{\det(\Sigma)}-1\right|=|\det(1+\Sigma^{-1}\Delta)-1|\le3\lVert\Sigma^{-1}\Delta\rVert_{\text{OP}}=O(e^{-2(1-\beta)C}).$$
Therefore, if $C>0$ is large enough,
$$\left|\log\dfrac{\det(\Sigma_{C+\theta})}{\det(\Sigma)}\right|\le2\left|\dfrac{\det(\Sigma_{C+\theta})}{\det(\Sigma)}-1\right|=O(e^{-2(1-\beta)C}),$$
since the ratio of the determinants can be arbitrarily close to 1.\\\\
We turn to the last term of the Right Hand Side of (\ref{KL-Gaus}). Observe that due to the initial condition and (\ref{norm-of-exp}),
$$\left\lVert e^{(C+\theta)A}z'-e^{(C+\theta)A}z\right\rVert_{2}\le\left\lVert e^{(C+\theta)A}\right\rVert_{\text{OP}}\cdot\lVert z-z'\rVert_2=O(\varepsilon).$$
This implies that
\begin{align*}
&\left|(e^{(C+\theta)A}z'-e^{(C+\theta)A}z)^\top\Sigma^{-1}(e^{(C+\theta)A}z'-e^{(C+\theta)A}z)\right|\\\le\ &\lVert\Sigma^{-1}\rVert_{\text{OP}}\cdot\left\lVert e^{(C+\theta)A}z'-e^{(C+\theta)A}z\right\rVert_{2}^2=O(\varepsilon^2).
\end{align*}
By Pinsker's inequality, we know that
$$\lVert P-Q\rVert_{\text{TV}}\le\sqrt{\dfrac{1}{2}D_{KL}(P||Q)},$$
so its application for the case of the Gaussian random vectors in question implies that
$$\left\lVert N\left(e^{(C+\theta)A}z',\Sigma_{C+\theta}\right)-N\left(e^{(C+\theta)A}z,\Sigma\right)\right\rVert_{\text{TV}}\le\sqrt{O(e^{-2(1-\beta)C})+O(\varepsilon^2)},$$
which is what we wanted to prove.\\\\
The second claim can be proven in the exact same way, with the only difference being that in this setup,
$$\left\lVert e^{(C+\theta)A}z'-e^{(C+\theta)A}z\right\rVert_2=O(e^{-(1-\beta)C}),$$
instead of $O(\varepsilon)$.\\\\
The third claim follows directly from (\ref{KL-Gaus}) and Pinsker's inequality.
\end{proof}
Next, we calculate the TV distances between the Gaussian random vectors that arise in this problem. First, we find $\Sigma$. It is well-known (see \cite{KS}, section 5.6 for details) that 
\begin{equation}\label{Lyapunov}
A\Sigma+\Sigma A^\top=-BB^\top
\end{equation}
and that this equation has a unique solution for each $B$ (see section 7.2 of \cite{Bhatia} for details). Once we show that the matrix
\begin{equation}\label{sigma-def}
\Sigma=\begin{pmatrix}
\frac{\lambda(1-(1-\lambda)\beta)}{1-\beta}&\frac{\lambda\beta(1-\lambda)}{1-\beta}\\\frac{\lambda\beta(1-\lambda)}{1-\beta}&\frac{(1-\lambda)(1-\lambda\beta)}{1-\beta}
\end{pmatrix}
\end{equation}
satisfies (\ref{Lyapunov}), this will indeed be the covariance matrix $\Sigma$. One can check that
$$A\Sigma=\Sigma A^\top=\begin{pmatrix}
-\lambda&0\\0&-(1-\lambda)
\end{pmatrix},$$
which clearly implies that (\ref{Lyapunov}) is satisfied. Also, observe that
\begin{align*}
A\begin{pmatrix}
\lambda\\1-\lambda
\end{pmatrix}&=\begin{pmatrix}
\lambda(-1+\lambda\beta)+(1-\lambda)\lambda\beta\\\lambda(1-\lambda)\beta+(1-\lambda)(-1+(1-\lambda)\beta)
\end{pmatrix}\\&=\begin{pmatrix}
-(1-\beta)\lambda\\-(1-\beta)(1-\lambda)
\end{pmatrix}=-(1-\beta)\begin{pmatrix}
\lambda\\1-\lambda
\end{pmatrix},
\end{align*}
so the vector $\begin{pmatrix}
    \lambda\\1-\lambda
\end{pmatrix}$
is an eigenvector for the eigenvalue $-(1-\beta)$ of $A$. Therefore, for any $\theta\in\R$, we can set
\begin{align}\label{mean}
\mu_\lambda(\theta)=e^{(C+\theta)A}\begin{pmatrix}
\alpha_1(\lambda,C)\\\alpha_2(\lambda,C)
\end{pmatrix}&=\alpha(\lambda,C)e^{(C+\theta)A}\begin{pmatrix}
    \lambda\\1-\lambda
\end{pmatrix}\notag\\&=g(2\lambda-1)e^{-(1-\beta)\theta+c(\beta)}\begin{pmatrix}
\lambda\\1-\lambda
\end{pmatrix}.
\end{align}
Because of this expression, it is clear that for any $\lambda$, the function $\theta\mapsto\mu_\lambda(\theta)$ is continuous.
\begin{lem}\label{Def-Psi}
For any $\theta\in\R, \lambda\in(1/2,1)$,
\begin{align*}
\lVert N(\mu_\lambda(\theta),\Sigma)-N(0,\Sigma)\rVert_{\text{TV}}&=\left\lVert N\left(g(2\lambda-1)\cdot e^{-(1-\beta)\theta+c(\beta)},\dfrac{1}{1-\beta}\right)-N\left(0,\dfrac{1}{1-\beta}\right)\right\rVert_{\text{TV}}\\&=:\Psi(\lambda,\theta).
\end{align*}
\end{lem}
\begin{rem}\label{cont-of-Psi}
Due to the definitions of $\Psi(\lambda,\theta)$ and $\Psi(\theta)$ and the last statement of Lemma \ref{small-tv},
\begin{enumerate}
\item
The function $\Psi:\R\to(0,1)$ is continuous.
\item 
For any $\theta\in\R$, $\lim\limits_{\lambda\uparrow1}\Psi(\lambda,\theta)=\Psi(\theta)$.
\end{enumerate}
\end{rem}
\begin{proof}[Proof of Lemma \ref{Def-Psi}]
For the proof of this Lemma, we will use the following formula for the total variation distance between two Gaussian random vectors with the same covariance matrix, that was proven in \cite{B-U}:
\begin{equation}\label{tv-gaus}
\lVert N(\mu,\Sigma)-N(0,\Sigma)\rVert_{\text{TV}}=2\Phi\left(\dfrac{\sqrt{\mu^\top\Sigma^{-1}\mu}}{2}\right)-1,
\end{equation}
where $\Phi$ is the standard CDF of a $N(0,1)$ random variable. In this case, from (\ref{sigma-def}), we have
$$\Sigma^{-1}=\begin{pmatrix}
\frac{1}{\lambda}-\beta&-\beta\\-\beta&\frac{1}{1-\lambda}-\beta
\end{pmatrix},$$
and
\begin{align*}
\Sigma^{-1}\mu_\lambda(\theta)&=g(2\lambda-1)e^{-(1-\beta)\theta+c(\beta)}\begin{pmatrix}
\frac{1}{\lambda}-\beta&-\beta\\-\beta&\frac{1}{1-\lambda}-\beta
\end{pmatrix}\begin{pmatrix}
\lambda\\1-\lambda
\end{pmatrix}\\&=(1-\beta)g(2\lambda-1)e^{-(1-\beta)\theta+c(\beta)}\begin{pmatrix}
    1\\1
\end{pmatrix}.
\end{align*}
Therefore, by (\ref{tv-gaus}),
\begin{align*}
\lVert N(\mu_\lambda(\theta),\Sigma)-N(0,\Sigma)\rVert_{\text{TV}}&=2\Phi\left(\dfrac{\sqrt{\mu_{\lambda}(\theta)^\top\Sigma^{-1}\mu_\lambda(\theta)}}{2}\right)-1\\&=2\Phi\left(\dfrac{g(2\lambda-1)e^{-(1-\beta)\theta+c(\beta)}\sqrt{1-\beta}}{2}\right)-1.
\end{align*}
It is easy to see that the formula for
$$\left\lVert N\left(g(2\lambda-1)\cdot e^{-(1-\beta)\theta+c(\beta)},\dfrac{1}{1-\beta}\right)-N\left(0,\dfrac{1}{1-\beta}\right)\right\rVert_{\text{TV}}$$
is identical, according to (\ref{tv-gaus}), which concludes the proof of this Lemma.
\end{proof}
\section{Proof of Theorem \ref{main-theorem}}
Before we move on to the proof of Theorem \ref{main-theorem}, we prove the following technical lemma, which will allow us to finish the proof of the upper bound, in the case that $|A|\in[n/2,(1-\varepsilon_0)n)$.
\begin{lem}\label{quick-coal}
For any $\varepsilon_0,\eta,c>0$, there exists some $\delta>0$ such that: If $|A|\in(\varepsilon_0 n,(1-\varepsilon_0)n)$ and two copies $\left(U_t^{(i)},V_t^{(i)}\right)$ (for $i=1,2$) of the two-coordinate chain start at $\left(u_0^{(i)},v_0^{(i)}\right)$, respectively such that
$$\left|u_0^{(1)}-u_0^{(2)}\right|, \left|v_0^{(1)}-v_0^{(2)}\right|\le\delta\sqrt{n}\ \ \text{and}\ \ \left|u_0^{(i)}\right|,\left|v_0^{(i)}\right|\le c\sqrt{n},$$
there is a way to couple the two chains so that if $n$ is large enough,
$$\PR\left(\left(U_{\delta+\delta^{1/2}}^{(1)},V_{\delta+\delta^{1/2}}^{(1)}\right)\neq\left(U_{\delta+\delta^{1/2}}^{(2)},V_{\delta+\delta^{1/2}}^{(2)}\right)\right)\le\eta.$$
\end{lem}
In the proof of Lemma \ref{quick-coal}, we will use certain well-known properties regarding the drift of the magnetization chain and the two-coordinate chain, as well as a hitting time property of the symmetric simple random walk. It should be noted that two of these properties were proven in \cite{LLP} for the discrete-time version of the Markov chain we are studying. The statements we provide here are adapted to our continuous-time version and their proofs are identical to the ones in discrete time.\\\\ 
At first we mention Lemma 2.3 from \cite{LLP}, in which it was proven that the magnetizations of two copies of the two-coordinate chain have a drift towards each other.
\begin{lem}[Lemma 2.3 of \cite{LLP}]\label{mag-neg-drift}
For the magnetization chain $(Y_t)_{t\ge0}$, the relation
$$q_n(k,k+2)+q_n(\ell,\ell-2)\ge q_n(k,k-2)+q_n(\ell,\ell+2)$$
holds for any $\ell\ge k$.
\end{lem}
We will use this lemma to prove that if we run the two-coordinate chains independently, the drift of the difference is negative. This, due to a standard hitting-time estimate of a symmetric simple random walk we mention next, will imply that the magnetizations become equal quickly with high probability. The lemma that will facilitate this argument is the following, and can be found as Corollary 2.28 in \cite{LPW}:
\begin{lem}\label{hit-0}
Let $(S_n)_{n\ge0}$ be a symmetric discrete-time simple random walk on $\Z$ starting at $S_0=k>0$ and let $\tau=\inf\{n\ge0:S_n=0\}$. Then for any $n\in\N$,
$$\PR(\tau>n)\le\dfrac{3k}{\sqrt{n}}.$$
\end{lem}
After the magnetizations become equal, the goal is to find a coupling of the two-coordinate chains that keeps the magnetizations equal and forces the two-coordinate chains to meet quickly with high probability. This coupling was described in Lemma 3.5 of \cite{LLP}.
\begin{lem}(Lemma 3.5 of \cite{LLP})\label{two-coord-neg-drift}
Let $\sigma_1,\sigma_2\in\{-1,1\}^n$ be configurations with $m(\sigma_1)=m(\sigma_2)$ and assume that $|A|\in(\varepsilon_0n,(1-\varepsilon_0)n)$. There exists a coupling of the Glauber dynamics $(X_t^{(1)})_{t\ge0}$, $(X_t^{(2)})_{t\ge0}$ with initial states $\sigma_1,\sigma_2$ respectively, such that:
\begin{enumerate}
\item
For every $t\ge0$, $m(X_t^{(1)})=m(X_t^{(2)})$.
\item
For every $t\ge0$, let $\tau_t$ be the first time of the first (common) update that happens after time $t$. Then,
$$\E\left[\left(U_{\tau_t}^{(2)}-U_{\tau_t}^{(1)}\right)-\left(U_t^{(2)}-U_t^{(1)}\right)|X_t^{(1)},X_t^{(2)}\right]\le0$$
and on the event that $\left|U_t^{(i)}\right|,\left|V_t^{(i)}\right|\le\varepsilon_0n/3$,
$$\E\left[\left(U_{\tau_t}^{(2)}-U_{\tau_t}^{(1)}\right)\neq\left(U_t^{(2)}-U_t^{(1)}\right)|X_t^{(1)},X_t^{(2)}\right]\ge c_0',$$
for some constant $c_0'$ that depends only on $\varepsilon_0$.
\end{enumerate}
\end{lem}
\begin{proof}[Proof of Lemma \ref{quick-coal}]
In the proof that follows, all stated or implied constants may depend on $\varepsilon_0,\eta,c$ but do not depend on $\delta$ or $n$. We couple the chains in two steps. First, we make sure that the two magnetizations $Y_t^{(i)}=U_t^{(i)}+V_t^{(i)}$ are equal with high probability at $t_1=\delta$ and then that $U_{t_0}^{(i)}$ become equal with high probability in the remaining time.\\\\
Without loss of generality, assume that $u_0^{(1)}+v_0^{(1)}\le u_0^{(2)}+v_0^{(2)}$. For the first part of the coupling, we let the two chains run independently until the magnetizations coalesce. Let 
$$\tau_1:=\inf\left\{t\ge0:Y_t^{(2)}=Y_t^{(1)}\right\}.$$
Due to Lemma \ref{mag-neg-drift}, the chain $\left(\frac{Y_t^{(2)}-Y_t^{(1)}}{2}\right)_{t\ge0}$ under the coupling we described above has non-positive drift. Also, for some $c_0>0$,
$$q_n(k,k+2)+q_n(k,k-2)=\dfrac{1}{2}n\left(1-\dfrac{k}{n}\tanh\left(\beta\cdot\dfrac{k}{n}\right)\right)+O(1)\ge c_0\cdot n.$$
Therefore, we can couple the difference chain with a symmetric $\pm1$-random walk $(S_t)_{t\ge0}$ on $\Z$ whose updates happen at rate $\ge c_0\cdot n$ and starts at $s=\left(u_0^{(2)}+v_0^{(2)}\right)-\left(u_0^{(1)}+v_0^{(1)}\right)$, so that $Y_t^{(2)}-Y_t^{(1)}\le2S_t$, for $0\le t\le \tau_1$. Let $R_\delta$ be the event that between times 0 and $\delta$, there have been at most $c_0\delta n/2$ updates. Note that $\PR(R_\delta)=o_{n\to\infty}(1)$ for any $\delta>0$, because of standard large deviation estimates. Because of Lemma \ref{hit-0},
\begin{equation}\label{tau_1}
\PR(\tau_1>\delta)\le\PR(S_t>0\ \forall t\in[0,\delta])\le\PR(R_\delta)+\dfrac{c_1\cdot s}{\sqrt{c_0\delta n/2}}=O(\delta^{1/2})+o_{n\to\infty}(1).\\\\
\end{equation}

We move on to the description and analysis of the second part of the coupling. In what follows, we condition on $U_\delta^{(i)},V_\delta^{(i)}$ and we assume, without loss of generality, that $U_\delta^{(2)}\ge U_\delta^{(1)}$. Due to Lemma \ref{tech-1},
\begin{equation}\label{diff-from-start}
\PR\left(\left|U_\delta^{(1)}-u_0^{(1)}\right|\ge\delta^{1/3}\sqrt{n}\right)=O(\delta^{1/3}),
\end{equation}
and the same holds for $U_\delta^{(2)},V_\delta^{(i)}$ (for $i=1,2$). Let 
\begin{align*}
E_1=\{\tau_1\le\delta\}&\cap\left\{\left|U_\delta^{(1)}-U_\delta^{(2)}\right|,\left|V_\delta^{(1)}-V_\delta^{(2)}\right|\le3\delta^{1/3}\sqrt{n}\right\}\\&\cap\left\{\left|U_\delta^{(1)}\right|,\left|U_\delta^{(2)}\right|,\left|V_\delta^{(1)}\right|,\left|V_\delta^{(2)}\right|\le(c+1)\sqrt{n}\right\}.
\end{align*}
Because of equations (\ref{tau_1}) and (\ref{diff-from-start}), we know that $\PR(E_1)\ge1-O(\delta^{1/3})-o_{n\to\infty}(1)$.  The coupling after time $\delta$ works as follows: On the event $E_1^c$, we run the two chains independently. On the event $E_1$, we run the coupling of Lemma \ref{two-coord-neg-drift}. Let
$$\tau_2:=\inf\left\{t\ge\delta:\max\left(\left|U_t^{(i)}\right|,\left|V_t^{(i)}\right|\right)>\varepsilon_0 n/3\right\}\ \ \text{and}\ \ \tau_3=\inf\left\{t\ge\delta:U_t^{(1)}=U_t^{(2)}\right\}.$$
By $E_2$ we denote the event that $\{\tau_2>\delta+\delta^{1/2}\}$. Lemma \ref{tech-2} implies that $\PR(E_2^c)=o_{n\to\infty}(1)$. Because of the properties of the coupling in Lemma \ref{two-coord-neg-drift}, we can couple the difference chain $\left(U_t^{(2)}-U_t^{(1)}\right)_{t\ge\delta}$ with a symmetric $\pm1$-random walk $(S_t')_{t\ge\delta}$ on $\Z$ whose updates happen at rate $\ge c_0'\cdot n$ and starts at $s'=U_\delta^{(2)}-U_\delta^{(1)}$, so that $U_t^{(2)}-U_t^{(1)}\le2S_t'$, for $\delta\le t\le\min(\tau_2,\tau_3)$. Just as before, let $R_\delta'$ be the event that between times $\delta$ and $\delta+\delta^{1/2}$, there have been at most $c_0'\delta^{1/2}n/2$ updates, and of course, $\PR(R_\delta')=o_{n\to\infty}(1)$ for any $\delta>0$. Under this coupling,
\begin{align*}
&\PR\left(\left(U_{\delta+\delta^{1/2}}^{(1)},V_{\delta+\delta^{1/2}}^{(1)}\right)\neq\left(U_{\delta+\delta^{1/2}}^{(2)},V_{\delta+\delta^{1/2}}^{(2)}\right)\right)\\\le\ &\PR(E_1^c)+\PR(E_2^c)+\PR(R_\delta')+\PR\left(S_t'>0\ \forall t\in[\delta,\delta+\delta^{1/2}]\ |\ E_1\cap E_2\right)\\\le\ &\PR(E_1^c)+\PR(E_2^c)+\PR(R_\delta')+\dfrac{c_1'\delta^{1/3}\sqrt{n}}{\sqrt{c_0'\delta^{1/2}n/2}}\\=\ &O(\delta^{1/12})+o_{n\to\infty}(1)<\eta,
\end{align*}
if $\delta$ is small enough and $n$ is large enough. We have proven the lemma.
\end{proof}
The first step towards proving Theorem \ref{main-theorem} is the following proposition.
\begin{prop}\label{first-step}
Let $(\sigma_n)_{n=1}^\infty$ be a sequence $\sigma_n\in\{-1,1\}^n$ with $|A_n|/n\to\lambda\in[1/2,1)$. Then, for every $\theta\in\R$
\begin{align*}
\lim\limits_{n\to\infty}\left\lVert H_{t_{n,\theta}}^{(n)}(\sigma_n,\cdot)-\pi_n\right\rVert_{\text{TV}}=\Psi(\lambda,\theta)<\Psi(\theta),
\end{align*}
where $\Psi(\lambda,\theta)$ was defined in Lemma \ref{Def-Psi}. For any $\theta\in\R$, we set $\Psi(1/2,\theta)=0$, as the formula for $\Psi(\lambda,\theta)$ suggests anyway.
\end{prop}
\begin{proof}
Throughout the proof of this Proposition, by a small abuse of notation, if $(Z_t)_{t\ge0}$ is a continuous-time process, we denote by $Z_\infty$ a random variable which follows the stationary distribution.\\\\ In the case $\lambda\in(1/2,1)$, fix $\eta>0$ and let $\varepsilon>0$, which will be chosen later. Due to Proposition \ref{conc-for-pair} if $C>0$ is large enough, then
\begin{equation}\label{conc-tcc}
\PR\left[\left(
\dfrac{U_{t_{n,-C}}}{\sqrt{n}},\dfrac{V_{t_{n,-C}}}{\sqrt{n}}\right)\in\mathcal{R}_{\varepsilon,C}\right]\ge1-\eta,
\end{equation}
where we set
$$\mathcal{R}_{\varepsilon,C}:=\left\{(u,v)\in\R^2:\lVert(u,v)-(\alpha_1(\lambda,C),\alpha_2(\lambda,C))\rVert_2\le\varepsilon\lVert(\alpha_1(\lambda,C),\alpha_2(\lambda,C))\rVert_2\right\}.$$
In the case $\lambda=1/2$, we again fix $\eta>0$ and set $M_2>0$ to be a constant for which the statement of Proposition \ref{conc-pair-1/2} holds. In this case, we set
$$\mathcal{R}_{M_2}:=B_\infty(0,M_2)\subseteq\R^2.$$
\textbf{Proof of the upper bound:} We deal with the case $\lambda\in(1/2,1)$. Every time the argument has to be modified in the case $\lambda=1/2$, we mention the modification. The notation in the two cases will differ, as in the $\lambda=1/2$ case, we find suitable parameters $M_2,C$ depending on $\eta$, instead of $\varepsilon, C$ and we work with $\mathcal{R}_{M_2}$ instead of $\mathcal{R}_{\varepsilon,C}$.\\\\ 
We claim that for any $\eta>0$, there exist $\varepsilon,C>0$ such that
$$\limsup\limits_{n\to\infty}\sup\limits_{(u,v)\in\sqrt{n}\mathcal{R}_{\varepsilon,C}}\left\lVert H_{C+\theta}^{(U,V)}((u,v),\cdot)-\pi_{U,V}\right\rVert_{\text{TV}}\le\Psi(\lambda,\theta)+6\eta.$$
Suppose the opposite holds. Then, there exists a sequence $(u_n,v_n)\in\sqrt{n}\mathcal{R}_{\varepsilon,C}$ such that
\begin{equation}\label{abs-1}
\limsup\limits_{n\to\infty}\left\lVert H_{C+\theta}^{(U,V)}((u_n,v_n),\cdot)-\pi_{U,V}\right\rVert_{\text{TV}}>\Psi(\lambda,\theta)+6\eta.
\end{equation}
Due to the compactness of $\mathcal{R}_{\varepsilon,C}$, we can assume (by extracting a subsequence and modifying the rest of the sequence, if necessary) that $(u_n,v_n)/\sqrt{n}\to z_0\in\mathcal{R}_{\varepsilon,C}$. Therefore, Proposition \ref{converg-diff} applies. Let $(D_t)_{t\ge0}$ be an Ornstein-Uhlenbeck process started at $D_0=(\alpha_1(\lambda,C),\alpha_2(\lambda,C))\in\R^2$ if $\lambda\in(1/2,1)$ or $D_0=(0,0)$ if $\lambda=1/2$. Also, we denote by $(D_t^{(0)})_{t\ge0}$ the Ornstein-Uhlenbeck process with $D_0^{(0)}=z_0$. Let $c>0$ be such that 
$$\PR\left[\lVert N(0,\Sigma)\rVert_{\infty}\ge c/2\right]\le\eta,$$
and $\varepsilon,C>0$ be such that the first statement of Lemma \ref{small-tv} holds for $\theta_0=|\theta|+1$. In the $\lambda=1/2$ case, $C>0$ will be chosen such that the second statement of Lemma \ref{small-tv} holds for $\theta_0=|\theta|+1$. Moreover, let $\delta>0$ be small enough so that the following two conditions hold:
\begin{itemize}
\item 
$\lVert N(\mu_\lambda(\theta-\delta-\delta^{1/2}),\Sigma)-N(\mu_\lambda(\theta),\Sigma)\rVert_{\text{TV}}\le\eta$.
\item 
The statement of Lemma \ref{quick-coal} is true, for the choices of $\eta,c$ made above. 
\end{itemize}
Such a choice for $\delta$ is possible due to the continuity of $\mu_\lambda$.\\\\
Set $\tilde{U}^{(n)}_0=u_n/\sqrt{n}, \tilde{V}^{(n)}_0=v_n/\sqrt{n}$. Due to Proposition \ref{converg-diff} and the Skorokhod's representation Theorem, if $n$ is large enough there exists a coupling of $(\tilde{U}_{C+\theta-\delta-\delta^{1/2}}^{(n)},\tilde{V}_{C+\theta-\delta-\delta^{1/2}}^{(n)})$ and $D_{C+\theta-\delta-\delta^{1/2}}^{(0)}$, such that
$$\PR\left[\left\lVert\left(\tilde{U}_{C+\theta-\delta-\delta^{1/2}}^{(n)},\tilde{V}_{C+\theta-\delta-\delta^{1/2}}^{(n)}\right)-D_{C+\theta-\delta-\delta^{1/2}}^{(0)}\right\rVert_\infty\le\delta/2\right]\ge1-\eta.$$
Similarly, there exists a coupling of $(\tilde{U}_{\infty}^{(n)},\tilde{V}_{\infty}^{(n)})$ and $D_{\infty}$, such that
$$\PR\left[\left\lVert\left(\tilde{U}_{\infty}^{(n)},\tilde{V}_{\infty}^{(n)}\right)-D_{\infty}\right\rVert_\infty\le\delta/2\right]\ge1-\eta.$$
Finally, due to Lemma \ref{small-tv} and the way $\delta$ was chosen, there exists a coupling of $D_{C+\theta-\delta-\delta^{1/2}}^{(0)}$ and $D_{\infty}$ such that
\begin{align*}
\PR\left[D_{C+\theta-\delta-\delta^{1/2}}^{(0)}=D_\infty\right]&=1-\left\lVert N(e^{(C+\theta-\delta-\delta^{1/2})A}z_0,\Sigma_{C+\theta-\delta-\delta^{1/2}})-N(0,\Sigma)\right\rVert_{\text{TV}}\\&\ge1-\left\lVert N(\mu_\lambda(\theta),\Sigma)-N(0,\Sigma)\right\rVert_{\text{TV}}-2\eta.
\end{align*}
Therefore, there exists a coupling of $\left(\tilde{U}^{(n)}_{C+\theta-\delta-\delta^{1/2}},\tilde{V}^{(n)}_{C+\theta-\delta-\delta^{1/2}}\right)$ and $\left(\tilde{U}^{(n)}_\infty,\tilde{V}^{(n)}_\infty\right)$ such that
$$\PR\left(\left\lVert\left(\tilde{U}^{(n)}_{C+\theta-\delta-\delta^{1/2}},\tilde{V}^{(n)}_{C+\theta-\delta-\delta^{1/2}}\right)-\left(\tilde{U}^{(n)}_\infty,\tilde{V}^{(n)}_\infty\right)\right\rVert_\infty\le\delta\right)\ge1-\Psi(\lambda,\theta)-4\eta.$$
Furthermore, due to the way $c>0$ was chosen,
$$\PR\left[\left\lVert\left(\tilde{U}^{(n)}_{C+\theta-\delta-\delta^{1/2}},\tilde{V}^{(n)}_{C+\theta-\delta-\delta^{1/2}}\right)\right\rVert_\infty\ \ \text{and}\ \ \left\lVert\left(\tilde{U}^{(n)}_\infty,\tilde{V}^{(n)}_\infty\right)\right\rVert_\infty\le c\right]\ge1-\eta.$$
We have found a coupling of $\left(\sqrt{n}\cdot\tilde{U}_{C+\theta-\delta-\delta^{1/2}}^{(n)},\sqrt{n}\cdot\tilde{V}^{(n)}_{C+\theta-\delta-\delta^{1/2}}\right)$ and $\left(\sqrt{n}\cdot\tilde{U}_\infty^{(n)},\sqrt{n}\cdot\tilde{V}_\infty^{(n)}\right)$, which makes them satisfy the hypotheses of Lemma \ref{quick-coal}, with probability $\ge1-\Psi(\lambda,\theta)-5\eta$. On this event, the coupling of Lemma \ref{quick-coal} fails to give coalesence of the two processes in time $\delta+\delta^{1/2}$ with probability at most $\eta$. This implies that there exists a coupling of $\left(\sqrt{n}\cdot \tilde{U}_{C+\theta}^{(n)},\sqrt{n}\cdot\tilde{V}^{(n)}_{C+\theta}\right)$ and $\left(\sqrt{n}\cdot\tilde{U}_\infty^{(n)},\sqrt{n}\cdot\tilde{V}_\infty^{(n)}\right)$ for which
$$\PR\left[\left(\sqrt{n}\cdot\tilde{U}_{C+\theta}^{(n)},\sqrt{n}\cdot\tilde{V}^{(n)}_{C+\theta}\right)=\left(\sqrt{n}\cdot\tilde{U}_\infty^{(n)},\sqrt{n}\cdot\tilde{V}_\infty^{(n)}\right)\right]\ge1-\Psi(\lambda,\theta)-6\eta.$$
This contradicts (\ref{abs-1}) and, therefore, proves our claim. To finish the proof of the upper bound, one has to observe that due to (\ref{conc-tcc}), for any $\eta>0$, there exist $\varepsilon,C>0$ such that
\begin{align*}
&\limsup\limits_{n\to\infty}\left\lVert H_{t_{n,\theta}}^{(n)}(\sigma_n,\cdot)-\pi_n\right\rVert_{\text{TV}}\\\le&\ \limsup\limits_{n\to\infty}\left(\PR\left[\left(
\dfrac{U_{t_{n,-C}}}{\sqrt{n}},\dfrac{V_{t_{n,-C}}}{\sqrt{n}}\right)\notin\mathcal{R}_{\varepsilon,C}\right]+\sup\limits_{(u,v)\in\sqrt{n}\mathcal{R}_{\varepsilon,C}}\left\lVert H_{C+\theta}^{(U,V)}((u,v),\cdot)-\pi_{U,V}\right\rVert_{\text{TV}}\right)\\\le&\ \Psi(\lambda,\theta)+7\eta.
\end{align*}
Since $\eta>0$ was arbitrary, we have successfully proven that
$$\limsup\limits_{n\to\infty}\left\lVert H_{t_{n,\theta}}^{(n)}(\sigma_n,\cdot)-\pi_n\right\rVert_{\text{TV}}\le\Psi(\lambda,\theta).$$
\textbf{Proof of the lower bound:} The need to prove the corresponding lower bound exists only in the case $\lambda\in(1/2,1)$. Note that $D_{C+\theta,1}+D_{C+\theta,2}$ is a Gaussian random variable with mean $g(2\lambda-1)\cdot e^{-(1-\beta)\theta+c(\beta)}$ and variance
\begin{align*}
\text{Var}(D_{C+\theta,1}+D_{C+\theta,2})&=\text{Var}(D_{C+\theta,1})+\text{Var}(D_{C+\theta,2})+2\text{Cov}(D_{C+\theta,1},D_{C+\theta,2})\\&\stackrel{(\ref{sigma-def})}{=}1+\dfrac{\beta}{1-\beta}=\dfrac{1}{1-\beta}.
\end{align*}
Set $\kappa_\theta=g(2\lambda-1)\cdot e^{-(1-\beta)\theta+c(\beta)}/2$ and observe that
\begin{align}\label{dist-stat}
\Psi(\lambda,\theta)=\ &\left\lVert N\left(g(2\lambda-1)\cdot e^{-(1-\beta)\theta+c(\beta)},\dfrac{1}{1-\beta}\right)-N\left(0,\dfrac{1}{1-\beta}\right)\right\rVert_{\text{TV}}\notag\\=\ &\PR\left[N\left(0,\dfrac{1}{1-\beta}\right)\le\kappa_\theta\right]-\PR\left[N\left(g(2\lambda-1)\cdot e^{-(1-\beta)\theta+c(\beta)},\dfrac{1}{1-\beta}\right)\le\kappa_\theta\right].
\end{align}
We next claim that, similarly as in the proof of the upper bound,
\begin{equation}\label{lb-int}
\liminf\limits_{n\to\infty}\inf\limits_{(u,v)\in\sqrt{n}\mathcal{R}_{\varepsilon,C}}\left\lVert H_{C+\theta}^{(U,V)}((u,v),\cdot)-\pi_{U,V}\right\rVert_{\text{TV}}\ge\Psi(\lambda,\theta)-2\eta.
\end{equation}
Let $(u_n,v_n)\in\sqrt{n}\mathcal{R}_{\varepsilon,C}$ be a sequence such that $(u_n,v_n)/\sqrt{n}\to z_0\in\mathcal{R}_{\varepsilon,C}$. We know that if we set $\tilde{U}_0^{(n)}=u_n/\sqrt{n}, \tilde{V}^{(n)}_0=v_n/\sqrt{n}$, then because of Proposition \ref{converg-diff},
$$(\tilde{U}^{(n)}_{C+\theta},\tilde{V}^{(n)}_{C+\theta})\xrightarrow[n\to\infty]{(d)}D_{C+\theta}^{(0)}\sim N(e^{(C+\theta)A}z_0,\Sigma_{C+\theta})\ \ \text{and}\ \ (\tilde{U}_\infty^{(n)},\tilde{V}_\infty^{(n)})\xrightarrow[n\to\infty]{(d)}D_\infty\sim N(0,\Sigma).$$
Because of this and Lemma \ref{small-tv},
\begin{align*}
\lim\limits_{n\to\infty}\PR\left(\tilde{U}_{C+\theta}^{(n)}+\tilde{V}_{C+\theta}^{(n)}\le\kappa_\theta\right)&=\PR\left(D_{C+\theta,1}^{(0)}+D_{C+\theta,2}^{(0)}\le\kappa_\theta\right)\\&\le\PR\left(D_{C+\theta,1}+D_{C+\theta,2}\le\kappa_\theta\right)+\eta\\&=\PR\left[N\left(g(2\lambda-1)\cdot e^{-(1-\beta)\theta+c(\beta)},\dfrac{1}{1-\beta}\right)\le\kappa_\theta\right]+\eta.
\end{align*}
Similarly,
\begin{align*}
\lim\limits_{n\to\infty}\PR\left(\tilde{U}_{\infty}^{(n)}+\tilde{V}_{\infty}^{(n)}\le\kappa_\theta\right)&=\PR\left(D_{\infty,1}+D_{\infty,2}\le\kappa_\theta\right)\\&\ge\PR\left[N\left(0,\dfrac{1}{1-\beta}\right)\le\kappa_\theta\right]-\eta.
\end{align*}
Choosing the set $(-\infty,\kappa_\theta\sqrt{n}]$ as the distinguishing statistic, we get that
\begin{align*}
&\left\lVert H_{C+\theta}^{(U,V)}((u_n,v_n),\cdot)-\pi_{U,V}\right\rVert_{\text{TV}}\ge\PR\left(U_\infty^{(n)}+V_\infty^{(n)}\le\kappa_\theta\sqrt{n}\right)-\PR\left(U_{C+\theta}^{(n)}+V_{C+\theta}^{(n)}\le\kappa_\theta\sqrt{n}\right)
\\&\stackrel{(\ref{dist-stat})}{\Rightarrow}\ \liminf\limits_{n\to\infty}\left\lVert H_{C+\theta}^{(U,V)}((u_n,v_n),\cdot)-\pi_{U,V}\right\rVert_{\text{TV}}\ge\Psi(\lambda,\theta)-2\eta.
\end{align*}
Working in a similar fashion as in the upper bound, we can finish the proof of (\ref{lb-int}). To conclude the proof of the lower bound, we observe that for any $\eta>0$, there exist $\varepsilon,C>0$ such that
\begin{align*}
&\liminf\limits_{n\to\infty}\left\lVert H_{t_{n,\theta}}^{(n)}(\sigma_n,\cdot)-\pi_n\right\rVert_{\text{TV}}\\\ge\ &\left(\PR\left[\left(
\dfrac{U_{t_{n,-C}}}{\sqrt{n}},\dfrac{V_{t_{n,-C}}}{\sqrt{n}}\right)\in\mathcal{R}_{\varepsilon,C}\right]\right)\cdot\inf\limits_{(u,v)\in\sqrt{n}\mathcal{R}_{\varepsilon,C}}\left\lVert H_{C+\theta}^{(U,V)}((u,v),\cdot)-\pi_{U,V}\right\rVert_{\text{TV}}\\\ge\ &(1-\eta)(\Psi(\lambda,\theta)-2\eta).
\end{align*}
Since $\eta>0$ was arbitrary, we have successfully proven that
$$\liminf\limits_{n\to\infty}\left\lVert H_{t_{n,\theta}}^{(n)}(\sigma_n,\cdot)-\pi_n\right\rVert_{\text{TV}}\ge\Psi(\lambda,\theta),$$
and the proof of Proposition \ref{first-step} has been concluded.
\end{proof}
\begin{cor}\label{first-step-sup}
For any $\varepsilon_0>0$,
$$\limsup\limits_{n\to\infty}\sup\limits_{\sigma:\  |m(\sigma)|\le(1-\varepsilon_0)n}\left\lVert H_{t_{n,\theta}}^{(n)}(\sigma,\cdot)-\pi_n\right\rVert_{\text{TV}}\le\Psi(\theta).$$
\end{cor}
\begin{proof}
If this statement is false, there exists some sequence of initial conditions $\sigma_n\in\{-1,1\}^n$ such that $|A_n|/n\to\lambda\in(\frac{\varepsilon_0}{2},\frac{2-\varepsilon_0}{2})$ and
$$\limsup\limits_{n\to\infty}\left\lVert H_{t_{n,\theta}}^{(n)}(\sigma_n,\cdot)-\pi_n\right\rVert_{\text{TV}}>\Psi(\theta).$$
We can, without loss of generality, due to the symmetry of the model, assume that $\lambda\ge1/2$. Observe that this set of conditions now directly contradicts Proposition \ref{first-step}.
\end{proof}
\begin{proof}[\textbf{Proof of Theorem \ref{main-theorem}}]
\textbf{Proof of the upper bound:} Assume, for the sake of contradiction, that for some $\theta\in\R$,
$$\limsup\limits_{n\to\infty}d_n(t_{n,\theta})>\Psi(\theta).$$
Then, there exists a sequence of configurations $\sigma_n\in\{-1,1\}^n$ such that
$$\limsup\limits_{n\to\infty}\left\lVert H_{t_{n,\theta}}^{(n)}(\sigma_n,\cdot)-\pi_n\right\rVert_{\text{TV}}>\Psi(\theta).$$
Let $A_n=\{i\in[n]:\sigma_n(i)=1\}$. Without loss of generality, due to the symmetry of the model, suppose that for any $n$, $|A_n|\ge n/2$. We can assume, by extracting a subsequence and modifying the rest of the sequence if necessary, that $|A_n|/n\to\lambda\in[1/2,1]$. The cases $\lambda\in[1/2,1)$ are not possible because of Proposition \ref{first-step}, so we only have to deal with the case $\lambda=1$. Let $\eta>0$. Choose $t_0>0$ such that $\Psi\left(\theta-t_0\right)<\Psi(\theta)+\eta$.
Such $t_0$ exists, as noted in Remark \ref{cont-of-Psi}.
Due to the estimates (\ref{first-mom}) and (\ref{sec-mom}) in the proof of Lemma \ref{1st-2nd-mom}, at time $t_0$,
$$\dfrac{\E(f_n(Y_{t_0})^2)}{\E(f_n(Y_{t_0}))^2}=\dfrac{e^{-2(1-\beta)t_0}f_n(m(\sigma_n))^2+O\left(\dfrac{1}{n}\right)}{e^{-2(1-\beta)t_0}f_n(m(\sigma_n))^2+O\left(\dfrac{1}{n}\right)}\ \ \Rightarrow\ \ \dfrac{\text{Var}(f_n(Y_{t_0}))}{\E(f_n(Y_{t_0}))^2}=O\left(\dfrac{1}{n}\right).$$
Due to this estimate, Chebyshev's inequality and the fact that $g(x)\ge x$ for $x>0$, if $n$ is large enough,
\begin{align*}
\PR\left[\dfrac{Y_{t_0}}{n}\notin\left(0,\dfrac{1+e^{-(1-\beta)t_0}}{2}\right)\right]&\le\PR\left[f_n(Y_{t_0})\notin\left(0,\dfrac{1+e^{-(1-\beta)t_0}}{2}\right)\right]\\&\le\PR\left[|f_n(Y_{t_0})-\E(f_n(Y_{t_0})|\ge\dfrac{1-e^{(1-\beta)t_0}}{3}\cdot\E(f_n(Y_{t_0})\right]
\\&\le\dfrac{9}{(1-e^{-(1-\beta)t_0})^2}\cdot\dfrac{\text{Var}(f_n(Y_{t_0}))}{\E(f_n(Y_{t_0}))^2}
\\&\le\eta.
\end{align*}
Therefore, if we set $\varepsilon_0=\frac{1-e^{-(1-\beta)t_0}}{2}>0$
\begin{align*}
&\limsup\limits_{n\to\infty}\left\lVert H_{t_{n,\theta}}^{(n)}(\sigma_n,\cdot)-\pi_n\right\rVert_{\text{TV}}\\\le\ &\limsup\limits_{n\to\infty}\left(\PR\left[\dfrac{Y_{t_0}}{n}\notin\left(0,\dfrac{1+e^{-(1-\beta)t_0}}{2}\right)\right]+\sup\limits_{\sigma:\  |m(\sigma)|\le(1-\varepsilon_0)n}\left\lVert H_{t_{n,\theta-t_0}}^{(n)}(\sigma,\cdot)-\pi_n\right\rVert_{\text{TV}}\right)
\\\le\ &\Psi(\theta-t_0)+\eta\le\Psi(\theta)+2\eta.
\end{align*}
Since $\eta>0$ was arbitrary, we have reached a contradiction.\\\\
\textbf{Proof of the lower bound:} 
Let $\eta>0$ and $\lambda\in(1/2,1)$ be such that $\Psi(\lambda,\theta)>\Psi(\theta)-\eta$. Such $\lambda$ exists, as noted in Remark \ref{cont-of-Psi}. Due to Proposition \ref{first-step}, for any sequence $\sigma_n\in\{-1,1\}^n$ with $|A_n|/n\to\lambda$,
$$\lim\limits_{n\to\infty}\left\lVert H_{t_{n,\theta}}^{(n)}(\sigma_n,\cdot)-\pi_n\right\rVert_{\text{TV}}=\Psi(\lambda,\theta)>\Psi(\theta)-\eta.$$
Therefore,
$$\liminf\limits_{n\to\infty}d_n(t_{n,\theta})\ge\Psi(\theta)-\eta,$$
and since $\eta>0$ was arbitrary, Theorem \ref{main-theorem} follows.
\end{proof}

\end{document}